\documentclass[12pt,reqno]{amsart}
\usepackage{amsmath,a4wide}
\usepackage{stmaryrd,mathrsfs,bm,amsthm,mathtools,yfonts,amssymb}
\usepackage{xcolor}

\begingroup
\newtheorem{theorem}{Theorem}[section]
\newtheorem{lemma}[theorem]{Lemma}
\newtheorem{proposition}[theorem]{Proposition}
\newtheorem{corollary}[theorem]{Corollary}
\endgroup

\theoremstyle{definition}
\newtheorem{definition}[theorem]{Definition}

\newtheorem{ipotesi}[theorem]{Assumption}

\numberwithin{equation}{section}

\title{Multiple valued functions and integral currents}
\author{Camillo De Lellis}
\author{Emanuele Spadaro}

\newcommand\supp{{\rm spt}}
\newcommand\bT{\mathbf{T}}

\newcommand\res{\mathop{\hbox{\vrule height 7pt width .3pt depth 0pt
\vrule height .3pt width 5pt depth 0pt}}\nolimits}
\newcommand{\im}{{\rm Im}}
\newcommand{\gr}{{\rm Gr}}
\newcommand{\bG}{\mathbf{G}}
\newcommand{\cH}{{\mathcal{H}}}
\newcommand{\bJ}{{\mathbf{J}}}
\newcommand{\cD}{{\mathcal{D}}}
\newcommand{\cone}{{\times\hspace{-0.6em}\times\,}}
\newcommand{\mass}{{\mathbf{M}}}
\newcommand{\de}{\partial}
\newcommand{\p}{{\mathbf{p}}}
\newcommand{\q}{{\mathbf{q}}}

\newcommand{\cM}{{\mathcal{M}}}
\newcommand{\bU}{{\mathbf{U}}}
\newcommand{\phii}{{\bm{\varphi}}}
\newcommand{\Phii}{{\bm{\Phi}}}

\newcommand{\bI}{{\bm{I}}}

\newcommand{\bF}{{\bm{F}}}

\newcommand{\bB}{{\mathbf{B}}}
\newcommand{\bC}{{\mathbf{C}}}
\newcommand\Z{{\mathbb Z}}

\newcommand\N{{\mathbb N}}

\newcommand\R{{\mathbb R}}
\newcommand{\eps}{{\varepsilon}}

\def\Xint#1{\mathchoice
{\XXint\displaystyle\textstyle{#1}}%
{\XXint\textstyle\scriptstyle{#1}}%
{\XXint\scriptstyle\scriptscriptstyle{#1}}%
{\XXint\scriptscriptstyle\scriptscriptstyle{#1}}%
\!\int}
\def\XXint#1#2#3{{\setbox0=\hbox{$#1{#2#3}{\int}$ }
\vcenter{\hbox{$#2#3$ }}\kern-.6\wd0}}
\def\mint{\Xint-}
\newcommand{\Iqs}{{\mathcal{A}}_Q(\R^{n})}
\newcommand{\Iq}{{\mathcal{A}}_Q}
\def\a#1{\left\llbracket{#1}\right\rrbracket}
\newcommand{\Lip}{{\rm {Lip}}}
\def\a#1{\left\llbracket{#1}\right\rrbracket}

\newcommand{\norm}[2]{\left\|#1\right\|_{#2}}
\newcommand{\cG}{{\mathcal{G}}}
\def\I#1{{\mathcal{A}}_{#1}}

\newcommand{\etaa}{{\bm{\eta}}}

\newcommand{\diam}{{\rm {diam}}}

\newcommand{\ph}{\varphi}

\begin{document}

\begin{abstract}
We prove several results on Almgren's multiple valued functions and
their links to integral currents.
In particular, we give a simple proof of the fact that
a Lipschitz multiple valued map naturally
defines an integer rectifiable current; we derive explicit formulae for
the boundary, the mass and the first variations along certain specific vector-fields;
and exploit this connection to derive a
delicate reparametrization property for multiple valued functions.
These results play a crucial role in our new proof of
the partial regularity of area minimizing currents
\cite{DS3,DS4,DS5}.
\end{abstract}

\maketitle

\section{Introduction}
It is known since the pioneering work of Federer and Fleming \cite{FF} that 
one can naturally associate an integer rectifiable current
to the graph of a Lipschitz function in the Euclidean space, integrating forms over the corresponding
submanifold, endowed with its natural orientation.
It is then possible to derive formulae for
the boundary of the current, its mass and its first variations along smooth
vector-fields. Moreover, all these formulae have important Taylor expansions when the current is sufficiently flat.
In this paper we provide elementary proofs for the corresponding facts
in the case of Almgren's multiple valued functions
(see \cite{DS1} for the relevant definitions).

The connection between multiple valued functions 
and integral currents
is crucial in the analysis of the regularity of area minimizing currents
for two reasons.
On the one hand, it provides the necessary tools for the approximation of currents
with graphs of multiple valued function.
This is a fundamental idea for the study of the regularity
of minimizing currents in the classical ``single-vaued'' case, and
it also plays a fundamental role in the proof of
Almgren's partial regularity result (cf.~\cite{Alm, DS3}). In this perspective,
explicit expressions for the mass and the first variations are
necessary to derive the right estimates on the main geometric quantities
involved in the regularity theory (cf.~\cite{DS3,DS4,DS5}).
On the other hand, the connection can be exploited 
to infer interesting conclusions about the multiple valued functions
themselves. 

This point of view has been taken fruitfully in many problems
for the case of classical functions (see, for instance,
\cite{GMbook1,GMbook2} and the references therein), 
and has been recently exploited in the multiple valued setting
in \cite{DFS, Sp10}.
The prototypical example of interest here is the following:
let $f: \R^m\supset \Omega\to \R^n$ be a
Lipschitz map and $\gr (f)$ its graph.
If the Lipschitz constant of $f$ is small
and we change coordinates in $\R^{m+n}$ with
an orthogonal transformation close to the identity, then the set $\gr (f)$ is the
graph of a Lipschitz function $\tilde{f}$ over 
some domain $\tilde{\Omega}$ also in the new system of coordinates.
In fact it is easy to see that there exist
suitable maps $\Psi$ and $\Phi$ such that
$\tilde{f} (x) = \Psi \big(x, f (\Phi (x))\big)$. 
In the multiple valued case, it remains still true that $\gr (f)$
is the graph of a new Lipschitz map $\tilde{f}$ in the new system of coordinates,
but we are not aware of any elementary proof of such statement, 
which has to be much more subtle because simple relations as the
one above cannot hold.
It turns out that the structure of $\gr (f)$ as integral current gives a simple
approach to this and similar issues.
Several natural estimates can then be proved for $\tilde{f}$, although
more involved and much harder.
The last section of the paper is dedicated to these questions;
more careful estimates obtained in the same vein will also be given
in \cite{DS4}, where they play a crucial role.

Most of the conclusions of this paper are already established, or
have a counterpart, in Almgren's monograph \cite{Alm},
but we are not always able to point out precise references to
statements therein.
However, also when this is possible, our proofs have an independent
interest and are in our opinion much simpler.
More precisely, the material of Sections \ref{s:currents} and \ref{s:boundaries}
is covered by \cite[Sections 1.5-1.7]{Alm}, where Almgren deals with
general flat chains. This is more than what is needed in \cite{DS3, DS4, DS5},
and for this reason we have chosen to treat only
the case of integer rectifiable currents. 
Our approach is anyway simpler and, instead of relying, as Almgren does,
on the intersection theory of flat chains, we use rather elementary tools.
For the theorems of Section~\ref{s:taylor_area} we cannot point out precise references, but Taylor expansions for the area functional
are ubiquitous in \cite[Chapters 3 and 4]{Alm}. The theorems
of Section~\ref{s:taylor_dT} do not appear in \cite{Alm}, as Almgren seems 
to consider only some particular classes of deformations (the ``squeeze'' and ``squash'', see \cite[Chapter 5]{Alm}), while we derive fairly general formulas. 
Finally, it is very likely that the conclusions
of Section~\ref{s:reparametrize} appear in some form in the construction
of the center manifold of \cite[Chapter 4]{Alm},
but we cannot follow the intricate arguments and notation of that chapter.
In any case, our approach to ``reparametrizions'' of multiple valued maps
seems more flexible and powerful, capable of further applications,
because, as it was first realized in \cite{DS1}, we can use tools from metric analysis
and metric geometry developed in the last 20 years.

\medskip

{\bf Acknowledgments} The research of Camillo De Lellis has been supported by the ERC grant
agreement RAM (Regularity for Area Minimizing currents), ERC 306247. 
The authors are warmly
thankful to Bill Allard for several enlightening conversations
and his constant enthusiastic encouragement; and very grateful to Luca Spolaor
and Matteo Focardi for carefully reading a preliminary version of the paper
and for their very useful comments.
Camillo De Lellis is also
very thankful to the University of Princeton, where he has spent most of his sabbatical completing
this and the papers \cite{DS3, DS4, DS5}.

\section{$Q$-valued push-forwards}\label{s:currents}

We use the notation $\langle , \rangle$ for: the euclidean
scalar product, the naturally induced inner products on $p$-vectors
and $p$-covectors
and the duality pairing of $p$-vectors and $p$-covectors;
we instead restrict the use of the symbol $\cdot$ to matrix products.
Given a $C^1$ $m$-dimensional 
submanifold $\Sigma\subset \R^N$, a function $f: \Sigma \to \R^k$ and a
vector field $X$ tangent to $\Sigma$, 
we denote by $D_X f$ the derivative of $f$ along $X$,
that is $D_X f (p) = (f\circ \gamma)' (0)$ 
whenever $\gamma$ is a smooth curve on $\Sigma$ with $\gamma (0) = p$ and $\gamma' (0) = X(p)$.
When $k=1$, we denote by $\nabla f$ 
the vector field tangent to $\Sigma$ such that $\langle \nabla f, X\rangle = D_X f$ for every tangent vector field $X$.
For general $k$, $Df|_x: T_x \Sigma \to \R^k$ will be the linear operator such that
$Df|_x \cdot X(x) = D_X f (x)$ for any tangent vector field $X$. We write $Df$ 
for the map $x\mapsto D f|_x$ and sometimes we will also use the notation $Df (x)$ in place of $Df|_x$.
Having fixed an orthonormal base $e_1, \ldots e_m$ on $T_x \Sigma$ and 
letting $(f_1, \ldots, f_k)$ be the components of $f$, we can write
$\nabla f_i = \sum_{j=1}^m a_{ij} e_j$ and $|Df|$ for the usual Hilbert-Schmidt norm:
\[
|Df|^2 = \sum_{j=1}^m |D_{e_j} f|^2 = \sum_{i=i}^k |\nabla f_i|^2 = \sum_{i,j} a_{ij}^2\, .
\]
All the notation above is extended to the differential of Lipschitz multiple
valued functions at points where they are differentiable in the sense of 
\cite[Definition 1.19]{DS1}: although the definition in there is for euclidean domains, 
its extension to $C^1$ submanifolds $\Sigma\subset \mathbb R^N$ is done, as usual, 
using coordinate charts.

We will keep the same notation also when $f=Y$ is a vector field, i.e. takes values in $\R^N$, the same
Euclidean space where $\Sigma$ is embedded. In that case we define additionally
${\rm div}_\Sigma Y := \sum_i \langle D_{e_i} Y, e_i\rangle$.
Moreover, when $Y$ is tangent to $\Sigma$, we introduce the covariant derivative
$D_\Sigma Y|_x$, i.e.~a linear map from $T_x \Sigma$ into itself which gives
the tangential component of $D_X Y$. Thus, if we denote by
$\p_x : \R^N \to T_x \Sigma$ the orthogonal projection onto $T_x \Sigma$, we have 
$D_\Sigma Y |_x = \p_x \cdot D Y (x)$. It follows that $D_\Sigma Y \cdot X = \nabla_X Y$,
where we use $\nabla$ for the connection (or covariant differentiation) on $\Sigma$ compatible with its structure 
as Riemannian submanifold of $\R^N$. Such covariant differentiation is then extended in 
the usual way to general tensors on $\Sigma$. 

When dealing with $C^2$ submanifolds $\Sigma$ of $\R^N$ we will denote by $A$ the 
following tensor: $A|_x$ as a bilinear map on $T_x \Sigma \times T_x \Sigma$ taking
values on $T_x \Sigma^\perp$ (the orthogonal complement of $T_x \Sigma$) and if 
$X$ and $Y$ are vector fields tangent to $\Sigma$, then $A (X, Y)$
is the normal component of $D_X Y$, which we will denote by $D^\perp_X Y$. $A$ is called
second fundamental form by some authors
(cf.~\cite[Section 7]{Sim}, where the tensor is denoted by $B$) and
we will use the same terminology, although in differential geometry it is more customary to call $A$ 
``shape operator'' and to use ``second fundamental form'' for scalar products
$\langle A(X,Y), \eta\rangle$ with a fixed normal vector field 
(cf. \cite[Chapter 6, Section 2]{DoCar} and \cite[Vol. 3, Chapter 1]{Spivak3}).
In addition,
$H$ will denote the trace of $A$ (i.e. $H = \sum_i A (e_i, e_i)$ where $e_1, \ldots, e_m$ is an orthonormal
frame tangent to $\Sigma$) and will be called {\em mean curvature}. 

\subsection{Push-forward through multiple valued functions of $C^1$ submanifolds}
In what follows we consider an $m$-dimensional $C^1$ submanifold $\Sigma$ of $\R^N$ and
use the word {\em measurable} for those subsets of $M$ which are $\cH^m$-measurable. Any time
we write an integral over (a measurable subset of) $\Sigma$ we understand that this integral
is taken with respect to the $\cH^m$ measure. We recall the following lemma which, even if not stated explicitely
in \cite{DS1}, is contained in several arguments therein.

\begin{lemma}[Decomposition]\label{l:chop}
Let $M \subset \Sigma$ be measurable and $F: M\to \Iqs$ Lipschitz. Then 
there are a countable partition of $M$ in bounded measurable subsets $M_i$ ($i\in \N$)
and Lipschitz functions $f^j_i: M_i\to \R^{n}$ ($j\in \{1, \ldots, Q\}$) such that
\begin{itemize}
\item[(a)] $F|_{M_i} = \sum_{j=1}^Q \a{f^j_i}$ for every $i\in \N$ and $\Lip (f^j_i)\leq
\Lip (F)$ $\forall i,j$;
\item[(b)] $\forall i\in \N$ and $j, j' \in \{1, \ldots ,Q\}$,
either $f_i^j \equiv f_i^{j'}$ or
$f_i^j(x) \neq f_i^{j'}(x)$ $\forall x \in M_i$;
\item[(c)] $\forall i$ we have $DF (x) = \sum_{j=1}^Q \a{Df_i^j (x)}$ for a.e. $x\in M_i$.
\end{itemize}
\end{lemma}
\begin{proof}
The proof is by induction on $Q$. For $Q=1$ it is obvious. Assume 
the statement for any $Q^*<Q$ and fix $F$ and $M$.
Note that, without loss of generality, we can assume that $M$ is bounded.
We set $M_0:= \{x: \exists \, y = y(x)\in \R^n \mbox{ with } F (x) = Q \a{y}\}$.
Clearly, $M_0$ is measurable
because it is the counterimage of a closed subset of $\Iqs$.
Moreover, $y:M_0\to \R^n$ is Lipschitz.
We then set $f_0^j = y$ for every $j\in \{1, \ldots, Q\}$. Next, consider $x\not\in M_0$.
By \cite[Proposition 1.6]{DS1} there exist a ball $B_x$,
two positive numbers $Q_1$ and $Q_2$, and
two Lipschitz $Q_l$-valued functions 
$g_l: M \cap B_x \to \mathcal{A}_{Q_l} (\R^{n})$ for $l=1,2$ such 
that $F|_{B_x\cap M} = \a{g_1} +\a{g_2}$.
We can apply the inductive hypothesis to $g_1$ and $g_2$, and conclude that
$F\vert_{B_x\cap M}$ can be reduced to the form as in (a) and (b) when restricted to a (suitably chosen) countable partition of $M\cap B_x$ into measurable sets. 
Since $\Sigma$ is paracompact, we can find a countable cover $\{B_{x_i}\}_i$ of $M\setminus M_0$, from which (a) and (b) follow.
The last statement can be easily verified at every Lebesgue point $x\in M_i$
where $F$ and all the $f_i^j$'s are differentiable. 
\end{proof}

When $F:M\subset\Sigma \to \R^n$ is a proper Lipschitz function and
$\Sigma\subset \R^N$ is oriented,
the current $S = F_\sharp \a{M}$ in $\R^n$ is given
by
\[
S (\omega) = \int_M \langle \omega(F(x)), DF(x)_\sharp \vec e(x)\,\rangle\, d\cH^m(x)
\quad \forall\;
\omega\in \cD^m(\R^n),
\]
where $\vec{e}(x) = e_1(x)\wedge \ldots \wedge e_m(x)$
is the orienting $m$-vector of $\Sigma$ and 
\[
DF(x)_\sharp \vec e = (DF|_x \cdot e_1) \wedge \ldots \wedge (DF|_x \cdot e_m),
\]
(cf.~\cite[Remark 26.21(3)]{Sim}; as usual $\cD^m (\Omega)$ denotes the space of smooth $m$-forms compactly supported in $\Omega$).
Using the Decomposition Lemma \ref{l:chop} it is
possible to extend this definition to multiple valued functions.
To this purpose, we give the definition of \textit{proper}
multiple valued functions. 

\begin{definition}[Proper $Q$-valued maps]\label{d:proper}
A measurable $F:M \to \Iqs$ is called
\textit{proper} if there is a measurable selection $F^1, \ldots, F^Q$
as in \cite[Definition 1.1]{DS1} (i.e. $F = \sum_i \a{F^i}$) such that
%$F_i:M \to \R^n$ is properfor every $i=1,\ldots, Q$, that is, 
$\bigcup_i \overline{(F^i)^{-1}(K)}$ is compact
for every compact $K \subset \R^n$. It is then obvious that if there
exists such a selection, then {\em every} measurable selection shares the
same property.
\end{definition}

We warn the reader that the terminology might be slightly misleading,
as the condition above is effectively {\em stronger} than the usual
properness of maps taking values in the metric space $(\Iqs, \cG)$, even
when $F$ is continuous: the standard notion of {\em properness} would
not ensure the well-definition of the multiple-valued push-forward.

\begin{definition}[$Q$-valued push-forward]\label{d:push_forward}
Let $\Sigma \subset\R^N$ be a $C^1$ oriented manifold, $M\subset \Sigma$ a
measurable subset and $F:M\to\Iqs$ a proper Lipschitz map. Then,
we define the push-forward of $M$ through $F$ as the current
$\mathbf{T}_F = \sum_{i,j} (f_i^j)_\sharp \a{M_i}$, where $M_i$ and $f_i^j$
are as in Lemma \ref{l:chop}: that is,
\begin{equation}\label{e:push_forward}
\mathbf{T}_F (\omega) := \sum_{i\in\N} \sum_{j=1}^Q \underbrace{\int_{M_i} 
\langle \omega (f_i^j(x)), Df_i^j (x)_\sharp \vec{e}(x)\,\rangle \, d\cH^m(x)}_{ T_{ij} (\omega)}
\quad \forall\; \omega \in \cD^m(\R^{n})\, .
\end{equation}
\end{definition}
We first want to show that $T$ is well-defined. Since $F$ is proper, we easily deduce that
\[
|T_{ij} (\omega)|\leq \Lip (F) \|\omega\|_\infty \cH^m ((f_i^j)^{-1}) (\supp (\omega)) <\infty.
\]
On the other hand, upon setting $F^j (x):= f^j_i (x)$ for $x\in M_i$, we have
$\cup _i (f_i^j)^{-1} (\supp (\omega)) = (F^j)^{-1} (\supp (\omega))$ and
$(f_i^j)^{-1} (\supp (\omega)) \cap (f_{i'}^j)^{-1} (\supp (\omega))=\emptyset$ for $i\neq i'$,
thus leading to
\begin{align*}
\sum_{i,j} |T_{ij} (\omega)| \leq \Lip(F)\,\|\omega\|_{\infty}\,
\sum_{j=1}^Q \cH^m((F^j)^{-1}(\supp(\omega))) < +\infty.
\end{align*}
Therefore, we can pass the sum inside the integral in 
\eqref{e:push_forward} and, by Lemma~\ref{l:chop}, get
\begin{equation}\label{e:repr_formula}
\mathbf{T}_F (\omega) = \int_M \sum_{l=1}^Q \langle \omega (F^l(x)), DF^l (x)_\sharp \vec{e}(x)\,\rangle
\, d\cH^m(x) \quad \forall\; \omega \in \cD^m(\R^{n}).
\end{equation}
In particular, recalling the standard theory of
rectifiable currents (cf.~\cite[Section 27]{Sim}) and the area formula (cf.~\cite[Section 8]{Sim}),
we have achieved the following proposition.

\begin{proposition}[Representation of the push-forward]\label{p:repr_formula}
The definition of the action of $\bT_F$
in \eqref{e:push_forward} does not depend on the chosen
partition $M_i$ nor on the chosen decomposition $\{f^j_i\}$, \eqref{e:repr_formula} holds and, hence, $\bT_F$ is a (well-defined)
integer rectifiable current given by $\bT_F = (\im(F),\Theta, \vec \tau)$ where:
\begin{itemize}
\item[(R1)] $\im(F)=\bigcup_{x\in M} \supp (F(x)) =
\bigcup_{i\in \N} \bigcup_{j=1}^Q f_i^j(M_i)$ is an $m$-dimensional rectifiable set; 
\item[(R2)] $\vec\tau$ is a Borel unitary $m$-vector orienting $\im (F)$;
moreover, for $\cH^m$-a.e. $p\in \im (F)$, we have 
$Df_i^j(x)_\sharp \vec e (x) \neq 0$ for every $i,j, x$ with $f_i^j (x) =p$ and
\begin{equation}\label{e:orientazione}
\vec\tau(p) = \pm
\frac{Df_{i}^j(x)_\sharp \vec e(x) }{|Df_i^j(x)_\sharp
\vec e (x) |}\, ;
\end{equation}
\item[(R3)] for $\cH^m$-a.e.~$p\in \im (F)$,
the (Borel) multiplicity function $\Theta$ equals
\[
\Theta(p) := \sum_{i,j,x : f_i^j(x) = p} \left\langle\vec\tau,
\frac{Df_{i}^j(x)_\sharp \vec e(x) }{|Df_i^j(x)_\sharp
\vec e (x) |}
\right\rangle.
% \cH^0\big(\big\{j : \exists\;i\in \N, \,x\in M_i
% \;\text{with}\;f_i^j(x) = p\big\}\big).
\]
\end{itemize}
\end{proposition}

\subsection{Push-forward of Lipschitz submanifolds}
As for the classical push-forward, Definition \ref{d:push_forward} can be extended to domains $\Sigma$
which are Lipschitz submanifolds using the fact that such $\Sigma$ can be ``chopped'' 
into $C^1$ pieces. Recall indeed the following fact.

\begin{theorem}[{\cite[Theorem 5.3]{Sim}}]\label{t:Rad-Whitney}
If $\Sigma$ is a Lipschitz $m$-dimensional oriented submanifold, 
then there are countably many $C^1$ $m$-dimensional oriented submanifolds $\Sigma_i$ which cover
$\cH^m$-a.s.~$\Sigma$ and such that the orientations of $\Sigma$ and $\Sigma_i$ coincide 
on their intersection.
\end{theorem} 

\begin{definition}[$Q$-valued push-forward of Lipschitz submanifolds]\label{d:push_lip}
Let $\Sigma \subset\R^N$ be a Lipschitz oriented submanifold, $M\subset \Sigma$ a
measurable subset and $F:M\to\Iqs$ a proper Lipschitz map. Consider the $\{\Sigma_i\}$ 
of Theorem \ref{t:Rad-Whitney} and set $F_i := F |_{M\cap \Sigma_i}$. Then,
we define the push-forward of $M$ through $F$ as the integer rectifiable current
$\mathbf{T}_F := \sum_i \mathbf{T}_{F_i}$.
\end{definition}

The aboved definition can be extended 
to $Q$-valued pushforwards of general rectifiable currents in a straightforward way: however this will never
be used in the papers \cite{DS3, DS4, DS5} and thus goes beyond the scope of our work. 
The following conclusion is a simple consequence of Theorem \ref{t:Rad-Whitney}
and classical arguments in geometric measure theory (cf.~\cite[Section 27]{Sim}).

\begin{lemma}\label{l:repr_for_Lip}
Let $M, \Sigma$ and $F$ be as in Definition \ref{d:push_lip} and consider
a Borel unitary $m$-vector $\vec{e}$ orienting $\Sigma$. Then $\mathbf{T}_F$ is
a well-defined integer rectifiable current for which all the conclusions of Proposition
\ref{p:repr_formula} hold.
\end{lemma}

As for the classical push-forward, $\mathbf{T}_F$ is invariant under
bilipschitz change of variables.

\begin{lemma}[Bilipschitz invariance]\label{l:biLipschitz_inv} Let $F: \Sigma\to \Iqs$ 
be a Lipschitz and proper map, $\Phi: \Sigma'\to \Sigma$ a
bilipschitz homeomorphism and $G:= F\circ \Phi$. Then, $\mathbf{T}_F = \mathbf{T}_G$.
\end{lemma}
\begin{proof} The lemma follows trivially from the corresponding result for classical 
push-forwards (see \cite[4.1.7 \& 4.1.14]{Fed}),
the Decomposition Lemma~\ref{l:chop} and the definition of $Q$-valued push-forward.
\end{proof}

We will next use
the area formula to compute explicitely the mass of $\bT_F$.
Following standard notation, we will
denote by $\bJ F^j (x)$ the Jacobian determinant of $DF^j$, i.e. the number
\[
\left|DF^j (x)_\sharp \vec{e}\,\right| = \sqrt{\det ((DF^j(x))^T \cdot DF^j (x))}
\]

\begin{lemma}[$Q$-valued area formula]\label{l:area}
Let $\Sigma, M$ and $F = \sum_j \a{F^j}$ be as in Definition \ref{d:push_lip}.
Then, for any bounded Borel function $h: \R^n \to [0, \infty[$, we have
\begin{gather}\label{e:mass}
\int h (p)\,  d \|\mathbf{T}_F\| (p) \leq \int_M \sum_j h (F^j (x))\, \bJ F^j (x) \, d\cH^m(x)  \, .
\end{gather}
Equality holds in \eqref{e:mass} if there is a set $M'\subset M$ 
of full measure for which 
\begin{equation}\label{e:no_cancellation}
\langle DF^j (x)_\sharp \vec{e} (x), DF^i (y)_\sharp \vec{e} (y) \rangle \geq 0 \qquad 
\forall x,y\in M' \; \mbox{and}\; i, j \;\mbox{with}\; F^i (x) = F^j (y)\, .
\end{equation}
If \eqref{e:no_cancellation} holds the formula is valid also for
bounded {\em real}-valued Borel $h$ with compact support.  
\end{lemma}

\begin{proof} Let $h: \R^n\to [0, \infty[$ be a Borel function.
Consider a decomposition as in the Decomposition Lemma \ref{l:chop}
and the integer rectifiable currents $T_{ij}$ of \eqref{e:push_forward}.
By the classical area formula, see \cite[Remark 27.2]{Sim}, we have
\begin{equation}\label{e:apezzi}
\int h (p)\,  d \|T_{ij}\| (p) \leq \int_{M_i} h (f^j_i (x)) \bJ f^j_i (x) \, d\cH^m(x).
\end{equation}
Summing this inequality over $i$ and $j$ and using Lemma \ref{l:chop}(c),
we easily conclude \eqref{e:mass}.
When \eqref{e:no_cancellation} holds, we can choose $\vec{\tau}$ of Proposition
\ref{p:repr_formula} such that the identity \eqref{e:orientazione} has always the $+$ sign. Define
$\Theta_{ij}(p) := \cH^0 (\{x: f^j_i (x) = p\}$. We then conclude from Proposition \ref{p:repr_formula}(R3)
that $\Theta (p) = \sum_{i,j} \Theta_{ij} (p)$ for $\cH^m$-a.e.~$p\in \im (F)$. 
On the other hand, again by \cite[Remark 27.2]{Sim},
equality holds in \eqref{e:apezzi} and, moreover, we have the identities
$\|T_{ij}\| = \Theta_{ij} \cH^m \res \im (f^j_i)$,
$\|\bT_F\| = \Theta \cH^m \res \im (F)$. This easily implies
the second part of the lemma and hence completes the proof.
\end{proof}

A particular class of push-forwards are given by graphs.

\begin{definition}[$Q$-graphs]\label{d:Q-graphs}
Let $\Sigma, M$ and $f= \sum_i \a{f_i}$ be as in Definition~\ref{d:push_lip}. Define 
the map $F: M \to \Iq(\R^{N+n})$ as $F (x):= \sum_{i=1}^Q \a{(x, f_i (x))}$. $\bT_F$ is 
the {\em current associated to the graph $\gr (f)$} and will be denoted by $\bG_f$. 
\end{definition}

Observe that, if $\Sigma$, $f$ and $F$ are as in Definition \ref{d:Q-graphs}, then
the condition \eqref{e:no_cancellation} is always trivially satisfied. Moreover, when
$\Sigma = \R^m$ the well-known Cauchy-Binet formula gives
\[
(\bJ F^j)^2 = 1 + \sum_{k=1}^m \sum_{A \in M^k (DF^j)} (\det A)^2\, ,
\]
where $M^k (B)$ denotes the set of all $k\times k$ minors of the matrix $B$. Lemma \ref{l:area}
gives then the following corollary in the case of $Q$-graphs

\begin{corollary}[Area formula for $Q$-graphs]\label{c:massa_grafico}
Let $\Sigma = \R^m$, $M\subset \R^m$ and $f$ be as in Definition \ref{d:Q-graphs}. Then,
for any bounded compactly supported Borel $h: \R^{m+n}\to \R$, we have
\begin{equation}\label{e:massa_grafico}
\int h(p)\, d\|\bG_f\| (p) = \int_M \sum_i h (x, f_i (x)) 
\Big(1 + \sum_{k=1}^m \sum_{A \in M^k (DF^j)} (\det A)^2\Big)^{\frac{1}{2}}\,dx .
\end{equation}
\end{corollary}

\section{Boundaries}\label{s:boundaries}
In the classical theory of currents, when $\Sigma$ is a Lipschitz manifold with Lipschitz boundary
and $F: \Sigma \to \R^N$ is Lipschitz and proper, then
$\partial (F_\sharp \a{\Sigma}) = F_{\sharp} \a{\partial \Sigma}$ (see \cite[4.1.14]{Fed}). 
This result can be extended to multiple-valued functions.

\begin{theorem}[Boundary of the push-forward]\label{t:commute}
Let $\Sigma$ be a Lipschitz submanifold of $\R^N$ with Lipschitz boundary,
$F:\Sigma \to \Iqs$ a proper Lipschitz function and $f= F\vert_{\de\Sigma}$.
Then, $\partial \mathbf{T}_F = \mathbf{T}_f$.
\end{theorem} 

The main building block is the following small variant of
\cite[Homotopy Lemma 1.8]{DS1}.

\begin{lemma}\label{l:hom}
There is $c (Q,m)>0$ such that, 
for every closed cube $C\subset\R^{m}$ centered at $x_0$
and every $F\in\Lip(C,\Iqs)$, we can find $G\in\Lip(C,\Iqs)$ satisfying:
\begin{itemize}
\item[(i)] $G\vert_{\de C}=F\vert_{\de C}=: f$, $\Lip(G)\leq c\,\Lip(F)$ and
$\norm{\cG(F,G)}{L^\infty}\leq c\,\Lip(F)\,\diam(C)$;
\item[(ii)] there are Lipschitz 
multi-valued maps $G_j$ and $f_j$ (with $j\in \{1, \ldots, J\}$) such that $G=\sum_{j=1}^{J}\a{G_j}$, 
$f=\sum_{j=1}^J\a{f_j}$ and $\mathbf{G}_{G_j}=\a{(x_0,a_j)}\cone \mathbf{G}_{f_j}$ for some $a_j\in\R^{n}$.
\end{itemize}
\end{lemma}

\begin{proof}
The proof of (i) is contained in \cite[Lemma 1.8]{DS1}.
Concerning (ii), the proof is contained in the inductive argument of
\cite[Lemma 1.8]{DS1}, it suffices to complement the arguments there with the
following fact:
if $C= [-1,1]^m$, $u\in \Lip (\partial C, \Iqs)$ and
$G (x)=\sum_i\a{\|x\| u_i\left(\frac{x}{\|x\|}\right)}$ is the ``cone-like''
extension of $u$ to $C$ (where $\|x\|=\sup_i |x_i|$), then
$\mathbf{G}_G=\a{0}\cone \mathbf{G}_u$. 
The proof of this claim is a simple consequence of  the Decomposition
Lemma~\ref{l:chop} and the very definition of $\mathbf{G}_u$. Consider, indeed, a countable
measurable partition $\cup_i M_i = \de C$ and Lipschitz functions $u^j_i$ with
$\sum_j \a{u^j_i} = u|_{M_i}$. According to our definitions, 
$\bG_u = \sum_{i,j} (u^j_i)_{\sharp} \a{M_i} =: \sum_{i,j} T_{ij}$. Consider now
for each $i$ the set $R_i := \{\lambda x:
x\in M_i, \lambda\in ]0,1]\}$ and define $G^j_i (\lambda x) := \lambda u^j_i (x)$ for every
$x\in M_i$ and $\lambda\in ]0,1]$. The sets $R_i$ are a measurable decomposition of
$C\setminus \{0\}$ and we have $\sum_j \a{G^j_i} = G|_{R_i}$.
Therefore, setting $S_{ij} := (G^j_i)_\sharp \a{R_i}$, we have
$\bG_G = \sum_{i,j} S_{ij}$. On the other hand, by the classical theory of currents
$S_{ij} = \a{0} \cone T_{ij}$ (see \cite[Section 4.1.11]{Fed}). Since 
$\sum_{ij} (\mass (S_{ij}) + \mass (T_{ij}))<\infty$, the desired claim follows.
\end{proof}

\begin{proof}[Proof of Theorem \ref{t:commute}] The proof is by induction on the dimension $m$.
Since every Lipschitz manifold can be triangulated and the statement is invariant
under bilipschitz homemorphisms, it suffices to prove the theorem when
$\Sigma = [0,1]^m$. Next, given a classical Lipschitz map
$\Phi: \R^N\to \R^k$, let
$\Phi\circ F$ be the multiple-valued map $\sum_i \a{\Phi (F_i)}$ 
(cf.~\cite[Section 1.3.1]{DS1}). If $F$ is a classical Lipschitz map, then $\bT_{\Phi\circ F}
= \Phi_\sharp F_\sharp \a{\Sigma} = \Phi_\sharp \bT_F$ (cf.~\cite[4.1.14]{Fed}). The same identity
holds for $Q$-valued map, as the Decomposition Lemma~\ref{l:chop} easily reduces it to the
single-valued case.
Then, if $\p: \R^m\times \R^{m+n}\to \R^{m+n}$ is the orthogonal
projection on the second components, we have 
$\p_\sharp \mathbf{G}_F = \mathbf{T}_F$. Given the classical
commutation of boundary and (single-valued) push-forward (see \cite[Section 4.1.14]{Fed}) we
are then reduced to proving he identity $\partial \bG_F = \bG_f$.

We turn therefore to the case $\mathbf{G}_F$. The starting step $m=1$ is an obvious
corollary of the Lipschitz selection principle 
\cite[Proposition 1.2]{DS1}.
Indeed, for $F \in \Lip([0,1],\Iqs)$, there exist functions $F_i\in \Lip([0,1],\R^n)$
such that $F = \sum _i \a{F_i}$.
Therefore, $\bT_F = \sum_i \bT_{F_i}$ and
\[
\de \bT_F = \sum _i \de \bT_{F_i} = \sum _i \left(\a{F_i(1)} - \a{F_i(0)}\right)
 = \bT_{f}.
\]
For the inductive argument, consider the dyadic decomposition at scale $2^{-l}$ of $[0,1]^m$:
\[
[0,1]^m=\bigcup_{k\in\{0,\ldots,2^l-1\}^m}Q_{k,l},\quad\text{with}\quad
Q_{k,l}=2^{-l}\left(k+[0,1]^m\right).
\]
In each $Q_{k,l}$, let $u_{k,l}$ be the cone-like extension
given by Lemma \ref{l:hom} of $f_{k,l} := F|_{\partial Q_{k,l}}$. Denote by $u_l$ the $Q$-
function on $[0,1]^m$ which coincides with 
$u_{k,l}$ on each $Q_{k,l}$. Obviously the 
$u_l$'s are equi-Lipschitz and converge uniformly to $F$ by Lemma \ref{l:hom} (i).
Set $T_l:=  \mathbf{G}_{u_l} = \sum_k \mathbf{G}_{u_{k,l}}$.
By the inductive hypothesis $\partial \mathbf{G}_{f_{k,l}} = 0$.
Since $\partial (\a{p} \cone T) = T - \a{p}\cone \partial T$
(see \cite[Section 26]{Sim}), Lemma \ref{l:hom} implies
$\partial \mathbf{G}_{u_{k,l}} = \mathbf{G}_{f_{k,l}}$.
Considering that the boundary faces common to adjacent cubes come with opposite orientations,
we conclude $\de T_l=\mathbf{G}_f$. 
By Corollary \ref{c:massa_grafico}, 
$\limsup_l (\mass (T_l)+\mass (\partial T_l)) < \infty$ and so
the compactness theorem for integral currents (see \cite[Theorem 27.3]{Sim})
guarantees the existence of an integral current $T$ which is the weak limit of a 
subsequence of $\{T_l\}$ (not relabeled).
It suffices therefore to show that:
\begin{itemize}
\item[(C)]  if $\Omega\subset \R^m$ is an open set and $u_l$ is a sequence of Lipschitz 
$Q$-valued maps on $\Omega$ such that 
$u_l$ converge uniformly to some $F$ and 
$T_l:=\bG_{u_l}$ converge to an integral current $T$, then $T=\bG_F$.
\end{itemize}
We will prove (C) by induction over $Q$: the case $Q=1$ is classical (see for instance
\cite[Theorem 2, Section 3.1 in Chapter 3]{GMbook1} and \cite[Proposition 2, Section 2.1 in Chapter 3]{GMbook1}).
We assume (C) 
holds for every $Q^*<Q$ and want to prove it for $Q$.  Fix a sequence as in (C). 
Clearly $T$ is supported in the rectifiable set $\gr (F)$. Fix an 
orthonormal basis $e_1, \ldots, e_m$ of $\R^m$ and extend it to an orthonormal basis 
of $\R^{m+n}$ with positive orientation. Set $\vec{e} = e_1\wedge \ldots \wedge e_m$. 
Thanks to the Lipschitz regularity of $F$, $\gr (F)$ can be oriented by $m$-planes $\vec{\tau}$ 
with the property that 
$\langle \vec{\tau}, \vec{e}\rangle \geq c>0$, where the constant $c$ depends on 
$\Lip (F)$. We have $T = (\gr(F),  \vec\tau, \bar\Theta)$ and 
$\bG_F = (\gr(F), \vec\tau, \Theta)$: we just need to show that $\Theta = \bar\Theta$ $\cH^m$-a.e. on $\gr(F)$.

As observed in Lemma~\ref{l:chop} there is a closed set $M_0$ and a Lipschitz function $f_0$ such that:
\begin{itemize}
\item $F(x) = Q\a{f_0 (x)}$ for every $x\in M_0$;
\item $F$ ``splits'' locally on $\Omega' = \Omega\setminus M_0$ into (Lipschiz) 
functions taking less than $Q$ values.
\end{itemize}
Using the induction hypothesis, it is trivial to verify that 
$T\res \Omega'\times \R^n = \bG_F \res \Omega'\times \R^n$. 
Thus we just need to show that $\bar\Theta (x,f_0(x)) = \Theta (x, f_0(x))$ for $\cH^m$-a.e. 
$x\in M_0$. Consider the orthogonal projection $\p: \R^{m+n}\to \R^m$. 
By the well-known formula for the pusforward of currents (see \cite[Lemma 4.1.25]{Fed}), 
we have $\p_\sharp T = \bar{\Theta}' \a{\Omega}$ and $\p_\sharp \bG_F = \Theta' \a{\Omega}$, where 
\[
\bar\Theta ' (x) = \sum_{(x,y)\in \gr (F)} \bar\Theta (x,y) \qquad \mbox{and}\qquad 
\Theta' (x) = \sum_{(x,y)\in \gr (F)} \Theta (x,y)\, .
\]
Therefore $\bar\Theta' (x) = \bar\Theta (x,f_0 (x))$ and $\Theta' (x) = \Theta (x, f_0 (x))$ for 
$\cH^m$-a.e. $x\in M_0$. On the other hand, by the definition of $\bG_F$ and the very same
formula for the push-forward (i.e.~\cite[Lemma 4.1.25]{Fed})
it is easy to see that $\p_\sharp \bG_F = Q \a{\Omega} = \p_\sharp T_l$. Since
$\p_\sharp T_l$ converges to $\p_\sharp T$, we conclude that 
$\Theta'\equiv Q \equiv \bar{\Theta}'$ $\cH^m$-a.e. on $\Omega$, which in turn implies
$\Theta (x, f_0 (x)) = \bar\Theta (x, f_0 (x))$ for a.e. $x\in M_0$. 
This completes the proof of the inductive step.
\end{proof}

\section{Taylor expansion of the area functional}\label{s:taylor_area}
In this section we compute the Taylor expansion of the
area functional in several forms.
To this aim, we fix the following notation and hypotheses.

\begin{ipotesi}\label{i:tripla_malefica}
We consider the following:
\begin{itemize}
\item[(M)] an open submanifold $\cM\subset \R^{m+n}$ of dimension $m$ with
$\cH^m (\cM)<\infty$, which is the graph of a function $\bm{\varphi}: 
\mathbb R^m\supset \Omega\to \mathbb R^n$ with $\|\bm{\varphi}\|_{C^{3}}\leq \bar{c}$; $A$ and
$H$ will denote, respectively,
the second fundamental form and the mean curvature of $\cM$;
\item[(U)] a regular tubular neighborhood $\bU$ of $\cM$, i.e. the set of points $\{x+y: x\in \cM, y\perp T_x \cM,
|y|<c_0\}$, where the thickness $c_0$ is sufficiently small so that
the nearest point projection $\p:\bU \to \cM$ is well defined and $C^2$; the thickness is supposed to be
larger than a fixed geometric constant;
\item[(N)] a $Q$-valued map $F: \cM \to \Iq (\R^{m+n})$ of the form
\[
\sum_{i=1}^Q \a{F_i (x)}= \sum_{i=1}^Q \a{x+ N_i (x)},
\]
where $N: \cM\to \Iq(\R^{m+n})$ satisfies
$x+ N_i (x)\in \bU$, $N_i (x) \perp T_x \cM$ for every $x$ and $\Lip (N) \leq \bar{c}$.
\end{itemize}
\end{ipotesi}

We recall the notation $\etaa \circ F := \frac{1}{Q} \sum_i F_i$, for every multiple
valued function $F = \sum_i\a{F_i}$.

\begin{theorem}[Expansion of $\mass (\bT_F)$] \label{t:taylor_area}
If $\cM$, $F$ and $N$ are as in Assumption \ref{i:tripla_malefica} and $\bar{c}$
is smaller than a geometric constant, then
\begin{align}
\mass (\mathbf{T}_F) = {}&  Q \, \cH^m (\cM) - Q\int_\cM \langle H, \etaa\circ N\rangle
+ \frac{1}{2} \int_\cM |D N|^2\nonumber\\
&+ \int_\cM \sum_i \Big(P_2 (x,N_i) +  P_3 (x, N_i, DN_i) + R_4 (x, DN_i)\Big),  \label{e:taylor_area}
%&\mass (\mathbf{T}_F) =  Q \cH^m (\cM) - Q\int_\cM \langle H, \etaa\circ N\rangle
%+ \frac{1}{2} \int_\cM |D N|^2\nonumber\\
%&\qquad+ \int_\cM \sum_i \Big(P_2 (x,N_i (x)) +  P_3 (x, N_i (x), DN_i (x)) + R_4 (x, DN_i (x))\Big)\, d\cH^m (x)\,,  \label{e:taylor_area}
\end{align}
where
% $H$ is the mean curvature vector of $\cM$, the trace of the second fundamental form
% $A$, and 
$P_2$, $P_3$ and $R_4$ are $C^1$ functions with the following properties:
\begin{itemize}
\item[(i)] $n\mapsto P_2(x, n)$ is a quadratic form on the normal bundle of $\cM$
satisfying
\begin{equation}\label{e:order_2}
|P_2 (x,n)|\leq C |A (x)|^2 |n|^2 \qquad\quad \forall \; x\in \cM, \;\forall\; n\perp T_x\cM;
\end{equation}
\item[(ii)] $P_3 (x, n, D)= \sum_i L_i (x,n) Q_i (x, D)$, where
$n\mapsto L_i (x,n)$ are linear forms on the normal bundle of $\cM$ and $D\mapsto Q_i (x,D)$ are quadratic forms on the space of ${(m+n)\times(m+n)}$-matrices,
satisfying
\begin{gather*}
|L_i (x,n)||Q_i (x, D)|\leq C |A(x)||n||D|^2 \qquad
\forall x\in \cM, \, \forall n\perp T_x\cM, \, \forall D\,;
%\quad \forall \; x\in \cM, \, n\perp T_x\cM, \, D\in \R^{(m+n)\times(m+n)};
\end{gather*}

%$\forall \; x\in \cM, \;\forall\; n\perp T_x\cM,\;\forall \; D\in \R^{(m+n)\times(m+n)}$;
\item[(iii)] $|R_4 (x,D)| = |D|^3 L (x,D)$, for some function $L$ with $\Lip (L)\leq C$,
which satisfies $L(x,0)=0$ for every $x\in \cM$
and is independent of $x$ when $A\equiv 0$.
\end{itemize}
Moreover, for any Borel function $h: \R^{m+n}\to \R$, 
\begin{equation}\label{e:taylor_aggiuntivo} 
\left| \int h\, d\|\bT_F\| - \int_{\cM} \sum_i h \circ F_i\right|
\leq C \int_{\cM} \Big(\sum_i |A| |h \circ F_i||N_i| + \|h\|_\infty (|DN|^2 + |A| |N|^2)\Big), 
\end{equation}
and, if $h (p) = g (\p (p))$ for some $g$, we have 
\begin{align}
 \left| \int h\, d\|\bT_F\| - Q \int_\cM (1- \langle H, \etaa\circ N\rangle + \textstyle{\frac{1}{2}}|DN|^2) \, g\right|
&\leq C \int_\cM \big(|A|^2 |N|^2 + |DN|^4\big) |g|\, .\label{e:taylor_aggiuntivo_2}
\end{align}
\end{theorem}

In particular, as a simple corollary of the theorem above, we have the following.

\begin{corollary}[Expansion of $\mass (\bG_f)$] \label{c:taylor_area}
Assume $\Omega\subset \R^m$ is an open set with bounded measure and $f:\Omega
\to \Iqs$ a Lipschitz map with $\Lip (f)\leq \bar{c}$. Then,
\begin{equation}\label{e:taylor_grafico}
\mass (\mathbf{G}_f) = Q |\Omega| + \frac{1}{2} \int_\Omega |Df|^2 + \int_\Omega \sum_i \bar{R}_4 (Df_i)\, ,
\end{equation}
where $\bar{R}_4\in C^1$ satisfies $|\bar{R}_4 (D)|= |D|^3\bar{L} (D)$ for $\bar{L}$ with $\Lip (\bar{L})\leq C$ and
$\bar L(0) = 0$.
\end{corollary}
\begin{proof} The corollary is reduced to Theorem \ref{t:taylor_area} by simply setting
$\cM = \Omega\times \{0\}$,
\[
N = \sum_i \a{N_i (x)} := \sum_i \a{(0, f_i (x))}
\quad \mbox{and}\quad
F(x) = \sum_i \a{F_i (x)} = \sum_i \a{(x, f_i (x))}\, .
\]
Since in this case $A$ vanishes, \eqref{e:taylor_area} gives precisely \eqref{e:taylor_grafico}.
\end{proof}

\begin{proof}[Proof of Theorem \ref{t:taylor_area}]
We will in fact prove the statement for $\mass (\mathbf{T}_{F|_V})$,
where $V$ is any Borel
subset of $\cM$. Under this generality, by the decomposition Lemma~\ref{l:chop},
it is enough to consider the case $F|_V = \sum_i G_i$, where each $
G_i = F_i|_V =x+ N_i|_V$ is a (one-valued!) Lipschitz map.
Next observe that \eqref{e:no_cancellation} obviously holds if 
$\bar{c}$ is sufficiently small. Therefore, 
\[
\mass (\mathbf{T}_{F|_V})
= \sum_i \mass ((F_i)_\sharp \a{V})\, ,
\] 
and, since $\etaa\circ N = \frac{1}{Q}\sum_i N_i$,
the formula \eqref{e:taylor_area} follows
from summing the corresponding identities
\begin{align}
\mass ((F_i)_\sharp \a{V})  ={}&\cH^m (V) + \int_V \langle H, N_i\rangle
+ \frac{1}{2} \int_V |D N_i|^2\nonumber\\
& + \int_V \Big(P_2 (x,N_i) +  P_3 (x, N_i, DN_i) + R_4 (x, DN_i)\Big)\, .\label{e:taylor_sheet}
\end{align}

To simplify the notation we drop the subscript $i$ in the proof of \eqref{e:taylor_sheet}.
Using the area formula, we have that
\[
\mass (F_\sharp \a{V}) = \int_V |DF_\sharp \vec{\xi}|\, d\cH^m \, ,
\]
where $\vec{\xi} = \xi_1\wedge \ldots \wedge \xi_m$ is the simple $m$-vector
associated to an orthonormal frame on $T\cM$. 
By simple multilinear algebra $|DF_\sharp \vec{\xi}| = \sqrt{\det M}$, where $M$ is the $m\times m$ matrix given by  
\begin{align}
M_{jk} &= \langle DF\cdot \xi_j, DF \cdot \xi_k\rangle
= \langle \xi_j + DN \cdot \xi_j, \xi_k + DN \cdot \xi_k \rangle\nonumber\\ 
&= \delta_{jk} + \underbrace{\langle DN \cdot \xi_j, \xi_k\rangle + \langle DN \cdot \xi_k, \xi_j \rangle}_{a_{jk}}
+ \underbrace{\langle DN \cdot \xi_j, DN\cdot \xi_k\rangle}_{b_{jk}}\label{e:decompone}\, .
\end{align}
Set $a = (a_{jk})$, $b= (b_{jk})$ and denote by $M_2 (a+b)$ and $M_3 (a+b)$, respectively, the sum of all $2\times 2$ and that of all $3\times3$ minors of
the matrix $(a+b)$; similarly denote by $R (a+b)$ the sum of all $k\times k$ minors with $k\geq 4$.
Then,
\begin{equation}\label{e:sviluppo1}
\det M = 1+ {\rm tr}\, (a+ b) + M_2 (a+b) + M_3 (a+b) + R(a+b)\, .
\end{equation}
Observe that
the entries of $a$ are linear in $DN$ and those of $b$ are quadratic.
Thus,
\begin{align}
M_2 (a+b) &= M_2 (a) + M_2 (b) + C_2 (a,b),\label{e:sviluppo2}\\
M_3 (a+b) &= M_3 (a) + C_4 (a,b),\label{e:sviluppo3}
\end{align}
where $C_2 (a,b)$ is a linear combination of terms of the form $a_{jk} b_{lm}$ and $C_4 (a,b)$ is a polynomial in
the entries of $DN$ satisfying the inequality $|C_4 (a,b)|\leq C |DN|^4$.
Recall the Taylor expansion $\sqrt{1+\tau} = 1+\frac{\tau}{2} - \frac{\tau^2}{8} + \frac{\tau^3}{16} + g (\tau)$,
where $g$ is an analytic function with $|g(\tau)|\leq |\tau|^4$. With the aid of
\eqref{e:sviluppo1}, \eqref{e:sviluppo2} and \eqref{e:sviluppo3} we reach the following conclusion:
\begin{align}\label{e:Taylorone}
|DF_\sharp \vec{\xi}| ={} & 1 + \frac{{\rm tr}\, (a+b) + M_2 (a) +C_2 (a,b) + M_3(a)}{2} +\notag\\
&\quad-\frac{({\rm tr}\, a)^2+2 \, {\rm tr}\, a\, {\rm tr}\, b + 2\, {\rm tr}\, a\, M_2 (a)}{8} + \frac{({\rm tr}\, a)^3}{16} + R_4,
\end{align}
where $R_4$ is an analytic function of the entries of $DN$ which satisfies $|R_4 (DN)|\leq C |DN|^4$.
Observe next that $
{\rm tr}\, b = \sum_k \langle DN \cdot \xi_k, DN\cdot \xi_k\rangle = |DN|^2$.
Moreover, 
\begin{align*}
\langle DN \cdot \xi_j, \xi_k \rangle &=
% \langle D_{\xi_j} N, \xi_k \rangle =
\nabla_{\xi_j} (\langle N, \xi_k \rangle) - \langle N, \nabla_{\xi_j} \xi_k\rangle
= - \langle N, A (\xi_j, \xi_k)\rangle.%\label{e:seconda_forma_1}\, .
\end{align*}
Thus, by the symmetry
of the second fundamental form, we have 
\begin{equation*}%\label{e:seconda_forma_2}
a_{jk} = - 2 \langle A (\xi_j, \xi_k), N\rangle \quad \mbox{and}\quad
{\rm tr}\, a = - 2 \langle H, N \rangle\, .
\end{equation*}
We then can rewrite 
\begin{align}\label{e:taylorone2}
|DF_\sharp \vec{\xi}| = {} & 1 - \langle H, N \rangle + \frac{|DN|^2}{2} + 
\underbrace{\frac{M_2 (a)}{2} - \frac{({\rm tr}\, a)^2}{8}}_{P_2} + \notag\\
&\quad + \underbrace{\frac{C_2 (a,b) + M_3(a)}{2} - 
\frac{\rm tr\, a\, ({\rm tr}\, b + M_2 (a))}{4} + \frac{({\rm tr}\, a)^3}{16}}_{P_3}
+ R_4\, .
\end{align}
Integrating \eqref{e:taylorone2} we reach \eqref{e:taylor_sheet}. It remains to show that $P_2$, $P_3$ and
$R_4$ satisfy (i), (ii) and (iii). If $A=0$, then $\cM$ is flat and the frame $\xi_1, \ldots, \xi_m$ can be
chosen constant, so that $R_4$ will not depend on $x$. 
Next, each $b_{jk}$ is a quadratic polynomial in the entries of $DN$, with coefficients which are $C^2$ functions of $x$.
Instead each $a_{jk}$ can be seen as a linear function in $DN$ with coefficients which
are $C^2$ functions of $x$, but also as a linear function $L_{jk}$ of $N$, with a $C^1$ dependence on $x$. In the latter case we have the bound $|L_{jk} (x,n)|\leq |A(x)||n|$.
Therefore the claims in (i) and (ii) follow easily.
Finally, since $R_4$ is an analytic function of the entries of $DN$ satisfying
$|R_4(DN)| \leq C \, |DN|^4$, the representation in (iii) follows from the elementary
consideration that $\frac{R_4(D)}{|D|^3}$ is a Lipschitz function vanishing at the origin.

Finally, observe that the argument above implies 
\eqref{e:taylor_aggiuntivo_2} when $g$ is the indicator function of any measurable set and the
general case follows from standard
measure theory.
The identity \eqref{e:taylor_aggiuntivo} follows easily from the same formulas for
$|DF_\sharp \vec{\xi}|$, using indeed cruder estimates.
\end{proof}

\subsection{Taylor expansion for the excess in a cylinder} The last results of this
section concerns estimates of the excess in different systems of coordinates,
in particular with respect to tilted planes and curvilinear coordinates.

\begin{proposition}[Expansion of a curvilinear excess]\label{p:eccesso_curvo}
There exist a dimensional constant $C>0$ such that, if
$\cM$, $F$ and $N$ are as in Assumption \ref{i:tripla_malefica} with $\bar{c}$ small enough, then
\begin{align}
\left\vert \int |\vec{\bT}_F (x) - \vec{\cM} (\p (x))|^2 \, d\|\bT_F\| (x)
- \int_\cM |DN|^2 \right\vert \leq  C \int_{\cM} (|A|^2 |N|^2 + |DN|^4)
\label{e:ecurvo_sopra}\,, 
\end{align}
where $\vec{\bT}_F$ and $\vec{\cM}$ are the unit $m$-vectors orienting $\bT_F$ and
$T\cM$, respectively.
\end{proposition}

\begin{proof}
Let $p\in \cM$ and define 
% Then, 
% \[
% |\vec{\bT}_F (p) - \vec{\cM} (x)|^2 = |\vec\tau - \vec\xi|^2
% \]
$\vec \cM(p)= \xi_1\wedge \ldots \wedge \xi_m$ for some orthonormal frame $\xi_1, \ldots, \xi_m$ for $T\cM$ and 
\[
\vec \bT_F(F_i (p)) = \textstyle{\frac{\vec\zeta_i}{|\vec\zeta_i|}} 
\qquad\mbox{with}\qquad \vec\zeta_i = (\xi_1+ DN_i|_p \cdot \xi_1)\wedge \ldots \wedge 
(\xi_m+ DN_i|_p \cdot \xi_m)\, .
\]
Our assumptions imply $\p (F_i (p))=p$. 
Using the $Q$-valued area formula and obvious computations we get
\[
\int |\vec{\bT}_F - \vec{\cM} \circ \p|^2 \, d\|\bT_F\| (x)
= \int_{\cM} \sum_i \left|\textstyle{\frac{\zeta_i}{|\zeta_i|}} - \vec\cM\right|^2|\zeta_i|
= \int_{\cM} 2 \Big(\sum_i |\zeta_i| - \sum_i \langle \zeta_i, \vec{\cM}\rangle\Big)\, .
\]
As already computed in the proof of Theorem \ref{t:taylor_area},
\[
\sum_i |\zeta_i| = Q - Q \langle H, \etaa\circ N \rangle + \frac{|DN|^2}{2} + O (|A|^2|N|^2 + |DN|^4)\, . 
\]
If we next define
$B^i_{jk} := \langle \xi_j, \xi_k + DN_i \cdot \xi_k\rangle = 
\delta_{jk} - \langle N_i, A (\xi_j, \xi_k)\rangle$, we then get
\[
\sum_i \langle \zeta_i, \vec\cM\rangle = \sum_i \det B^i = Q - Q \langle H, \etaa \circ N\rangle
+ O (|A|^2|N|^2)\, .
\]
Hence the claimed formula follows easily.
\end{proof}

Next we compute the excess of a Lipschitz graph with respect to a tilted plane.

\begin{theorem}[Expansion of a cylindrical excess]\label{t:taylor_tilt}
There exist dimensional constants $C, c>0$ with the following property.
Let $f: \R^m \to \Iqs$ be a Lipschitz map with ${\rm Lip}\, (f)\leq c$.
For any $0<s$, set $A:= \mint_{B_s} D (\etaa \circ f)$ and denote by $\vec\tau$
the oriented unitary $m$-dimensional simple vector to the graph of the linear map
$y\mapsto A\cdot y$. Then, we have
\begin{equation}\label{e:taylor_tilt}
\left| \int_{\bC_s} \left| \vec{\bG}_f - \vec\tau\right|^2\, d \|\bG_f\| - \int_{B_s} \cG (D f, Q \a{A})^2 \right| \leq C \int_{B_s} |Df|^4\, .
\end{equation}
\end{theorem}

\begin{proof}[Proof of Theorem \ref{t:taylor_tilt}]
Arguing as in the previous proofs, thanks to Lemma~\ref{l:chop},
we can write $f= \sum_i \a{f_i}$ and process local computations (when needed) as if each $f_i$ were Lipschitz. 
Moreover, we have that
\[
\vec\tau = \textstyle{\frac{\vec\xi}{|\xi|}} \quad \text{with }\;
\vec\xi = (e_1+ A \,e_1)\wedge \ldots \wedge (e_m+A\,e_m).
\]
Here and for the rest of this proof, we identify $\R^m$ and $\R^n$ with
the subspaces $\R^m \times \{0\}$ and $\{0\} \times \R^n$ of $\R^{m+n}$, respectively:
this justifies the notation $e_j+A\,e_j$ for $e_j \in \R^m$ and $A\, e_j \in \R^n$.
Next, we recall that
\[ 
|\xi| = \sqrt{\langle \xi, \xi\rangle} = \sqrt{\det (\delta_{ij} + \langle A \, e_i , A\, e_j\rangle)} = 1 + \textstyle{\frac{1}{2}} |A|^2
 +O (|A|^4).
%= 1 + \frac{|A|^2}{2} + O (({\rm Lip}\, f)^4)
\]
By Corollary~\ref{c:massa_grafico} we also have
\begin{align}
E &:= \int_{\bC_s} \left| \vec{\bG}_f - \vec\tau\right|^2\, d \|\bG_f\| = 2\,\mass (\bG_f) - 2 \int \langle \vec\bG_f, \vec\tau \,\rangle\, d\|\bG_f\|\nonumber\\
&= 2 \, Q \, |B_s| + \int_{B_s} (|Df|^2 + O (|Df|^4)) - 2\int \sum_i \langle (e_1+Df_i\, e_1)\wedge \ldots \wedge (e_m+Df_i \, e_m), \vec\tau \,\rangle\nonumber.
\end{align}
On the other hand $\langle A\,e_j , e_k\rangle = 0 = \langle Df_i\, e_j , e_k\rangle$. Therefore,
\begin{align*}
\langle (e_1+Df_i \, e_1)\wedge& \ldots \wedge (e_m+Df_i\,e_m), \vec \tau \, \rangle
= |\xi|^{-1} \det (\delta_{jk} +
\langle Df_i \, e_j , A \,e_k\rangle)\nonumber\\
&=
\left(1 + \frac{|A|^2}{2} + O (|A|^4)\right)^{-1}
 \left(1 + Df_i : A + O (|Df|^2|A|^2)\right) \, .
\end{align*}
Recalling that $|A|\leq C s^{-m} \int |Df| \leq C \left(s^{-m} \int |Df|^4\right)^{\frac{1}{4}}$, we then conclude
\begin{align*}
E & = \int_{B_s} |Df|^2 +  Q\,|B_s|\,|A|^2 - 2 \int_{B_s} \sum_i Df_i :A  + O \left(\int_{B_s} |Df|^4\right)\\ 
%&= \int_{B_s} \sum_k \left( |Df_k|^2 - 2 Df_k : A + |A|^2\right) + O ((\Lip f)^4)\\
&=\int_{B_s} \sum_i |Df_i - A|^2 + O\left(\int_{B_s} |Df|^4\right) = \int_{B_s} \cG (Df, Q\a{A})^2 + O \left(\int_{B_s} |Df|^4\right)\, .\qedhere
\end{align*}
\end{proof}

\section{First variations}\label{s:taylor_dT}
In this section we compute the first variations of the currents induced by multiple
valued maps. These formulae are ultimately the link between the stationarity
of area minimizing currents and the partial differential equations satisfied by suitable approximations.
We use here the following standard notation: 
given a current $T$ in $\R^N$ and a vector field
$X \in C^1 (\R^N,\R^N)$, we denote the first variation of $T$ along $X$ by
$\delta T (X) := \left.\frac{d}{dt}\right\vert_{t=0} \mass({\Phi_t}_\sharp T)$,
where $\Phi: ]-\eta, \eta[\times U \to \R^N$ is any $C^1$ isotopy of a neighborhood
$U$ of $\supp (T)$ with $\Phi (0, x)=x$ for any $x\in U$ and
$\left.\frac{d}{d\eps}\right\vert_{\eps=0} \Phi_\eps= X$
(in what follows we will often use
$\Phi_\eps$ for the map $x\mapsto \Phi (\eps, x)$). 
It would be more appropriate
to use the notation $\delta T (\Phi)$ (see, for instance, \cite[Section 5.1.7]{Fed}), 
but since the currents considered in this
paper are rectifiable, it is well known that the first variation depends only on $X$
and is given by the formula
\begin{equation}\label{e:first_var}
\delta T (X) = \int {\rm div}_{\vec{T}}\, X\, d\|T\|,
\end{equation}
where ${\rm div}_{\vec{T}}\, X = \sum_i \langle D_{e_i} X, e_i\rangle$ for any orthonormal frame
$e_1, \ldots, e_m$ with $e_1\wedge \ldots \wedge e_m = \vec{T}$ (see \cite[5.1.8]{Fed} and cf.~\cite[Section 2.9]{Sim}).
We begin with the expansion for the first variation of graphs. In what follows, $A:B$ will denote the usual Hilbert Schmidt
scalar product of two $k\times j$ matrices.

\begin{theorem}[Expansion of $\delta \bG_f (X)$]\label{t:grafici}
Let $\Omega\subset \R^m$ be a bounded open set and $f:\Omega \to \Iqs$ a map 
with $\Lip (f)\leq \bar{c}$.
Consider a function $\zeta\in C^1(\Omega \times \R^m, \R^n)$ and the corresponding vector field
$\chi\in C^1 (\Omega\times \R^n, \R^{m+n})$ given by $\chi (x,y) = (0, \zeta (x,y))$.
Then,
\begin{equation}\label{e:variazione_media}
\left|\delta \mathbf{G}_f (\chi) - \int_\Omega\sum_i  \big(D_x \zeta (x, f_i) + D_y \zeta (x, f_i) \cdot Df_i\big) : Df_i \right| 
\leq C \int _\Omega |D\zeta| |Df|^3\, .
\end{equation}
\end{theorem}

The next two theorems deal with general
$\bT_F$ as in Assumption \ref{i:tripla_malefica}. However we restrict our attention to ``outer and inner variations'',
where we borrow our terminology from the elasticity theory and the literature on harmonic maps. Outer variations result
from deformations of the normal bundle of $\cM$ which are the identity on $\cM$ and 
map each fiber into itself, whereas inner variations result from composing
the map $F$ with isotopies of $\cM$.

\begin{theorem}[Expansion of outer variations]\label{t:outer}
Let $\cM$, $\mathbf{U}$, $\p$ and $F$ be as in Assumption \ref{i:tripla_malefica}
with $\bar{c}$ sufficiently small. If $\varphi\in C^1_c (\cM)$ and $X(p) := 
\varphi (\mathbf{p} (p)) (p-\mathbf{p}(p))$, then
\begin{align}
\delta \mathbf{T}_F (X) = \int_\cM \Big(\varphi \, |DN|^2 + 
\sum_i (N_i \otimes D \varphi) : DN_i\Big) - \underbrace{Q \int_\cM \varphi \langle H, \etaa\circ N\rangle}_{{\rm Err}_1} + \sum_{i=2}^3{\rm Err}_i
% + {\rm Err}_2 + {\rm Err}_3\,
\label{e:outer} 
\end{align}
where
\begin{gather}
|{\rm Err}_2| \leq C \int_\cM |\varphi| |A|^2|N|^2\label{e:outer_resto_2}\\
|{\rm Err}_3|\leq C \int_\cM \Big(|\varphi| \big(|DN|^2 |N| |A| + |DN|^4\big) +
|D\varphi| \big(|DN|^3 |N| + |DN| |N|^2 |A|\big)\Big)\label{e:outer_resto_3}\, .
\end{gather}
\end{theorem}

Let $Y$ be a $C^1$ vector field on $T\cM$ with compact
support and define $X$ on $\bU$ setting $X (p) = Y (\p (p))$. Let $\{\Psi_\eps\}_{\eps \in ]-\eta, \eta[}$ be
any isotopy with $\Psi_0 = {\rm id}$ and $\left.\frac{d}{d\eps}\right\vert_{\eps=0}
\Psi_\eps = Y$ and define the following
isotopy of $\mathbf{U}$: $
\Phi_\eps (p) = \Psi_\eps (\mathbf{p} (p)) + (p-\mathbf{p} (p))$.
Clearly $X = \left.\frac{d}{d\eps}\right|_{\eps =0} \Phi_\eps$. 

\begin{theorem}[Expansion of inner variations]\label{t:inner}
Let $\cM$, $\mathbf{U}$ and $F$ be as in Assumption \ref{i:tripla_malefica} with $\bar{c}$ sufficiently small.
If $X$ is as above, then
\begin{align}
\delta \mathbf{T}_F (X) = \int_\cM \Big( \frac{|DN|^2}{2} {\rm div}_{\cM}\, Y -
\sum_i  D N_i : ( DN_i\cdot D_{\cM} Y)\Big)+ \sum_{i=1}^3{\rm Err}_i,\label{e:inner} 
\end{align}
where
\begin{gather}
{\rm Err}_1 = - Q \int_{ \cM}\big( \langle H, \etaa \circ N\rangle\, {\rm div}_{\cM} Y + \langle D_Y H, \etaa\circ N\rangle\big)\, ,\label{e:inner_resto_1}\allowdisplaybreaks\\
|{\rm Err}_2| \leq C \int_\cM |A|^2 \left(|DY| |N|^2  +|Y| |N|\, |DN|\right), \label{e:inner_resto_2}\allowdisplaybreaks\\
|{\rm Err}_3|\leq C \int_\cM \Big( |Y| |A| |DN|^2 \big(|N| + |DN|\big) + |DY| \big(|A|\,|N|^2 |DN| + |DN|^4\big)\Big)\label{e:inner_resto_3}\, .
\end{gather}
\end{theorem}

\subsection{Proof of Theorem \ref{t:grafici}} Set  $\Phi_\eps (x,y) := (x, y + \eps \,\zeta (x, y))$. 
For $\eps$ sufficiently small $\Phi_\eps$ is a  diffeomorphism of $\Omega \times \R^n$ into intself.
Moreover, $\left.\frac{d}{d\eps} \Phi_\eps \right|_{\eps =0} = \chi$.
Let $f_\eps = \sum_i \a{f_i + \eps\,\zeta (x, f_i)}$. Since 
$(\Phi_ \varepsilon)_\sharp \mathbf{G}_f = \mathbf{G}_{f_\varepsilon}$, we can
apply Corollary \ref{c:taylor_area} to compute
\begin{align}
\delta \mathbf{G}_f (\chi) 
&
%= \left. \frac{d}{d\eps} \mass ((\Phi_\eps)_\sharp \mathbf{G}_f)\right|_{\eps =0} 
= \frac{d}{d\eps}\Big|_{\eps =0}
\mass (\mathbf{G}_{f_\eps})
\stackrel{\eqref{e:taylor_grafico}}{=} \frac{d}{d\eps}\Big|_{\eps=0} 
\frac{1}{2} \int \sum_i \left(|D (f_i +\eps \,\zeta)|^2 + \bar{R}_4 (D (f_i+\eps\,\zeta))\right) \nonumber\allowdisplaybreaks\\
&= Q \int \sum_i  \big(D_x \zeta (x, f_i) + D_y \zeta (x, f_i) \cdot Df_i\big) : Df_i + \int 
\frac{d}{d\eps}\Big|_{\eps =0} \bar{R}_4
(Df_i + \eps D\zeta).\nonumber
\end{align}
Since $\bar{R}_4 (M) = |M|^3 L (M)$ for some Lipschitz $L$ with $L(0)=0$,
we can estimate as follows:
\[
\left|\frac{d}{d\eps}\Big|_{\eps =0} \bar{R}_4 (M+\eps \zeta)\right| \leq C L(M) |M|^2 |D\zeta| + 
C |M|^3 \Lip (L) |D\zeta| \leq C |M|^3|D\zeta|\, ,
\]
thus concluding the proof.

\subsection{Proof of Theorem \ref{t:outer}} Consider the map $\Phi_\eps (p) = p +\eps X (p)$. If $\eps$ is 
sufficiently small,
$\Phi_\eps$ maps $\mathbf{U}$ diffeomorphically in a neighborhood of $\cM$ and we obviously have
$\delta \mathbf{T}_F (X) = \left.\frac{d}{d\eps} \mass ((\Phi_\eps)_\sharp \mathbf{T}_F)\right|_{\eps=0}$.
Next set
$F_\eps (x) = \sum_i \a{x + N_i (x) (1+\eps \, \varphi (x))}$ and observe that
$(\Phi_\eps)_\sharp \mathbf{T}_F = \mathbf{T}_{F_\eps}$. Thus we can apply
Theorem \ref{t:taylor_area} to get:
\begin{align}
\delta \mathbf{T}_F (X) = & \int_\cM \Big(\varphi \, |DN|^2 
+ 
\sum_i (N_i \otimes D \varphi) : DN_i\rangle
\Big)
 -\underbrace{\int_\cM Q\,\ph\,\langle H, \etaa \circ N \rangle}_{=:{\rm Err}_1}
\nonumber\allowdisplaybreaks\\
&+ \underbrace{\int_\cM \sum_i \frac{d}{d\eps}\Big|_{\eps =0} P_2 (x, N_i (1+\eps\varphi))}_{=:{\rm Err}_2}\nonumber\allowdisplaybreaks\\
&+ \underbrace{\int_\cM \sum_i \frac{d}{d\eps}\Big|_{\eps =0} \big(P_3 (x, N_i (1+\eps \varphi), D (N_i (1+\eps \varphi)))+ R_4 (x, D (N_i (1+\eps \varphi)))\big)
}_{=:{\rm Err}_3}\, .\nonumber
\end{align}
Since $n\mapsto P_2 (x, n)$ is a quadratic form, we have
$P_2 (x, N_i (1+\eps \varphi)) = (1+\eps \varphi)^2 P_2 (x, N_i)$ and thus \eqref{e:outer_resto_2}
follows from \eqref{e:order_2}. Next, by Theorem \ref{t:taylor_area}(ii), we have the bound
\begin{align*}%\label{e:cubico}
\left|\frac{d}{d\eps}\Big|_{\eps =0} P_3 (x, N_i (1+\eps \varphi), D (N_i (1+\eps \varphi)))\right|
\leq C |A(x)| \left(|D\ph|\,|N_i|^2 |DN_i|+ |\ph|\,|N_i||DN_i|^2\right)\, .
\end{align*}
Finally, taking into account Theorem \ref{t:taylor_area}(iii):
\begin{align*}
\left| \frac{d}{d\eps}\Big|_{\eps =0}R_4 (x, D (N_i (1+\eps \varphi)))\right|
\leq &C \left(|DN_i|^3 + |DN_i|^3 \Lip (L)\right) (|N_i||D\varphi| + |DN_i| \varphi)\, .%\label{e:quartico}
\end{align*}
Putting together the last two inequalities we get
\eqref{e:outer_resto_3}.

\subsection{Proof of Theorem \ref{t:inner}} Set $F_\eps (x)= \sum_i \a{x+ N_i (\Psi_\eps^{-1} (x))}$. 
Clearly, ${\Phi_\eps}_\sharp \bT_F = \bT_{F_\eps}$.
Fix an orthonormal frame $e_1, \ldots, e_m$ on
$T\cM$ and let $\vec{e}= e_1\wedge \ldots \wedge e_m$. By Lemma~\ref{l:area},
\[
\delta \mathbf{T}_F (X) = \frac{d}{d\eps} \Big|_{\eps=0}\mass (\mathbf{T}_{F_\eps})
= \frac{d}{d\eps}\Big|_{\eps =0} \int_\cM \sum_i |(DF_{\eps,i})_\sharp \vec{e}\,|\,.
\]
Fix $i\in \{1, \ldots, Q\}$. Using the chain rule \cite[Proposition~1.12]{DS1}, we have:
\[
(DF_{\eps,i})_\sharp \vec{e} = w_1 (\eps, x) \wedge \ldots \wedge w_m (\eps, x) =: \vec{w} (\eps, x),
\]
where $w_j (\eps, x) = e_j (x) + \left. DN_i\right|_{\Psi^{-1}_\eps (x)} \cdot \left.D\Psi^{-1}_\eps \right|_x \cdot e_j (x)$.
Set $v_j (\eps, x) = w_j (\eps, \Psi_\eps (x))$. Since $\Psi_0$ is the identity, we obviously have $\vec{v} (0, \cdot) = {DF_i}_\sharp \vec{e}$. 
If we denote by $\bJ\Psi_\eps (x)$ the Jacobian determinant of the
transformation $\Psi_\eps$, we can change variable in the integral to conclude:
\begin{align*}
\frac{d}{d\eps}\Big|_{\eps = 0} 
\int_\cM |(DF_{\eps,i})_\sharp \vec{e}\,| =&
\frac{d}{d\eps}\Big|_{\eps = 0}
\int_{\cM} |\vec{v} (\eps, x)| \bJ\Psi_\eps (x)\\
=& \int_{\cM} |(DF_i)_\sharp \vec{e}\,| \frac{d}{d\eps}\Big|_{\eps = 0} \bJ\Psi_\eps +
\int_{\cM} |\vec{v} (0,x)|^{-1} \langle \partial_\eps \vec{v} (0,x), \vec{v} (0,x)\rangle\\ 
=& \underbrace{\int_{\cM} |(DF_i)_\sharp \vec{e}\,|\,  {\rm div}_\cM\, Y}_{I_{i,1}} + 
\underbrace{\int_{\cM} \langle \partial_\eps \vec{v} (0,x), (DF_i)_\sharp \vec{e}\,\rangle}_{I_{i,2}} \\
& + 
\underbrace{\int_{\cM} \langle \partial_\eps \vec{v} (0,x), (DF_i)_\sharp \vec{e}\,\rangle \big(|{DF_i}_\sharp \vec{e}\,|^{-1}-1\big)}_{I_{i,3}}.
\end{align*}
Thus, $\delta \mathbf{T}_F (X) = \sum_i I_{i,1} + \sum_i I_{i,2} + \sum_i I_{i,3} =: I_1 + I_2 + I_3$ and
we will next estimate these three terms separately.

\medskip

\noindent{\it Step 1. Estimate on $I_1$.} By the $Q$-valued area formula of Lemma \ref{l:area}
and \eqref{e:taylor_aggiuntivo_2} in Theorem~\ref{t:taylor_area},
\begin{align*}
I_1 &= Q\int_{\cM} {\rm div}_{\cM}\, Y + \frac{1}{2} \int_\cM |DN|^2 {\rm div}_{\cM} Y - Q \int \langle H, \etaa\circ N \rangle 
{\rm div}_{\cM} Y + {\rm Err}
\end{align*}
where $|{\rm Err}|\leq C \int_{\cM} \left(|A|^2 |N|^2 + |DN|^4\right) |{\rm div}_\cM\, Y|$.
Since $\int_\cM {\rm div}_\cM\, Y =0$ (recall that $Y\in C^1_c (\cM)$), we easily conclude that
\begin{equation}\label{e:I_1}
I_1 =  \frac{1}{2} \int_\cM |DN|^2 {\rm div}_{\cM} Y - Q \int \langle H, \etaa\circ N\rangle\, {\rm div}_\cM Y + \sum_{j=2}^3{\rm Err}_j,
\end{equation}
where the ${\rm Err}_j$'s satisfy the estimates \eqref{e:inner_resto_2} and
\eqref{e:inner_resto_3}.

\medskip

\noindent{\it Step 2. Estimate on $I_2$.} Set
\[
\zeta_i (x) := \langle \partial_\eps \vec{v} (0,x), (DF_i)_\sharp \vec{e}\rangle = \langle \partial_\eps \vec{v} (0,x), \vec{v} (0,x)\rangle = \frac{1}{2} \left.\frac{d}{d\eps}\right|_{\eps=0} 
|\vec{v} (\eps,x)|^2\, .
\]
Since $|\vec{v} (\eps, x)|^2$ is independent of the orthonormal frame chosen,
having fixed a point $x\in\cM$, we can impose $D_{\cM} e_j = 0$ at $x$. By multilinearity
\begin{equation}\label{e:multi}
\partial_\eps \vec{v} (0,x) = \sum_j v_1 (0,x)\wedge \ldots \wedge \partial_\eps v_j (0,x) \wedge \ldots 
\wedge v_m (0,x)\, .
\end{equation}
 We next compute
\begin{align}
\partial_\eps v_j (0,x) &= \frac{\partial}{\partial \eps}\Big|_{\eps=0} e_j (\Psi_\eps (x))
+ \left.DN_i\right|_x \cdot \left(\frac{\partial}{\partial \eps}\Big|_{\eps=0} \left(\left.D\Psi_\eps^{-1}\right|_{\Psi_\eps (x)} \cdot e_j (\Psi_\eps (x))\right)\right)\nonumber\\
&= D_Y e_j (x)+ \left.DN_i\right|_x \cdot [Y, e_j] (x)\, ,\label{e:Lie!}
\end{align}
where $[Y, e_j]$ is the Lie bracket. On the other hand,
since $D_{\cM} e_j (x)=0$, we have $D_Y e_j (x) =  A (e_j, Y)$ and
$[Y, e_j] (x) = - \nabla_{e_j} Y (x)$. 
Recall that $v_j (0,\cdot) = e_j + DN_i \cdot e_j$. By the usual computations in multilinear algebra,
it turns out that $\zeta_i = \sum_j \det M^j$, where the entries of 
the $m\times m$ matrix $M^j$ are given by:
\begin{align*}
M^j_{\alpha\beta} &= \langle e_\alpha + DN_i \cdot e_\alpha, e_\beta + DN_i\cdot e_\beta\rangle= \delta_{\alpha\beta} + O (|A||N|) + O (|DN|^2)\quad \text{for }
\beta\neq j,\\%\label{e:entries_1}\\
M^j_{\alpha j} &= \langle e_\alpha + DN_i \cdot e_\alpha,  A(e_j, Y) - DN_i \cdot \nabla_{e_j} Y\rangle\, .% \label{e:entries_2}
\end{align*}
(The entries for $\alpha\neq j$ are computed as in the proof of Theorem \ref{t:taylor_area}). 
Denote by ${\rm Min}^j_{\alpha j}$ the $(m-1)\times (m-1)$ minor which is obtained by deleting
the $\alpha$ row and the $j$ column. We then easily get the following estimates:
\begin{gather}
\left|{\rm Min}^j_{\alpha j}\right|  \leq C (|DN|^2 + |A||N|) 
\quad \mbox{for $\alpha\neq j$}, \label{e:minori_1}\\
{\rm Min}^j_{jj}\;\; = 1 + O (|DN|^2 + |A||N|).\label{e:minori_2}
\end{gather}
Moreover, observe that
\begin{align}\label{e:minori_3}
M^j_{\alpha j}&= - \langle DN_i\cdot e_\alpha, DN_i \cdot \nabla_{e_j} Y\rangle -
\langle e_\alpha, DN_i\cdot \nabla_{e_j} Y\rangle + \langle A(e_\alpha, Y), DN_i\cdot e_j\rangle\nonumber\\
&= -  \langle DN_i\cdot e_\alpha, DN_i \cdot \nabla_{e_j} Y\rangle + \langle A( e_\alpha, \nabla_{e_j} Y), N_i\rangle
+ \langle A(e_\alpha, Y), DN_i\cdot e_j\rangle\, . 
\end{align}
We therefore conclude from \eqref{e:minori_1}, \eqref{e:minori_2} and \eqref{e:minori_3} that
\begin{align}
&\zeta_i (x) =\sum_j \det M^j = \sum_j \sum_\alpha (-1)^{j+\alpha} M^j_{\alpha j} {\rm Min}^j_{\alpha j}\nonumber\\
=&  \sum_j \left( - \langle DN_i\cdot e_j, DN_i \cdot \nabla_{e_j} Y\rangle + \langle A( e_j, \nabla_{e_j} Y), N_i\rangle
+ \langle A(e_j, Y), DN_i\cdot e_j\rangle\right)\nonumber\\
& + O \Big(|DY|\big(|DN|^4 + |A|^2|N|^2\big) + |Y| (|A||DN|^3 + |A|^2|N||DN|)\Big).\label{e:zeta_i}
\end{align}
Summing over $i$ and integrating, we then achieve
\begin{align}\label{e:I_2_quasi}
I_2 =& - \int_\cM \sum_i DN_i : (DN_i \cdot D_{\cM} Y) + J_2 + {\rm Err}_2 + {\rm Err}_3\, ,
\end{align}
where ${\rm Err}_2$, ${\rm Err}_3$ are estimated as in \eqref{e:inner_resto_2}, \eqref{e:inner_resto_3}, and
\[
J_2 = Q \int_\cM \sum_j \left(
\langle A (e_j, \nabla_{e_j} Y), \etaa\circ N\rangle + \langle A( e_j, Y), D_{e_j} \etaa \circ N\rangle\right).
\]
In order to treat this last term, we consider the vector field
$Z = \sum_j \langle A (e_j, Y), \etaa\circ N \rangle\, e_j$. $Z$ is independent of the choice of the orthonormal frame $e_j$: therefore, to compute its divergence at a specific point $x\in\cM$ we can assume $D_{\cM} e_j =0$. We then get
\begin{equation*}
{\rm div}_{\cM} Z = \sum_j \left(\langle A( e_j, Y), D_{e_j} \etaa \circ N\rangle +
\langle D^\perp_{e_j} A (e_j, Y), \etaa \circ N\rangle + 
\langle A (e_j, \nabla_{e_j} Y), \etaa \circ N\rangle\right)\, ,
\end{equation*}
where the tensor $D^\perp_X A (U,Y)$ is defined as
\[
(D_X (A(U,Y)))^\perp - A (\nabla_X U, Y) - A (U, \nabla_X Y),
\]
(recall that $(D_X W)^\perp$ denotes the normal component of $D_X W$).  
The Codazzi-Mainardi equations (cf. \cite[Chapter 7.C, Corollary 15]{Spivak4}
imply the symmetry of $D^\perp A$. Thus,
\begin{equation}\label{e:codazzi}
\sum_j \langle D^\perp_{e_j} A (e_j, Y), \etaa \circ N\rangle =
\sum_j \langle D^\perp_Y A (e_j, e_j), \etaa \circ N \rangle = \langle D_Y^\perp H, \etaa\circ N\rangle.
\end{equation}
Summarizing (and recalling that $\etaa\circ N$ is normal to $\cM$),
\begin{equation}\label{e:Gauss}
{\rm div}_{\cM} Z = \sum_j \left(\langle A( e_j, Y), D_{e_j} \etaa \circ N\rangle + 
\langle A (e_j, \nabla_{e_j} Y), \etaa \circ N\rangle\right) + \langle D_Y H, \etaa \circ N \rangle\, .
\end{equation}
Since $Z$ is compactly supported in $\cM$, integrating \eqref{e:Gauss} and using the
divergence theorem we conclude 
$0 = Q^{-1} J_2+\int \langle D_Y H, \etaa\circ N\rangle$. 
We thus get
\begin{align*}%\label{e:I_2}
I_2 = - \int_\cM \sum_i DN_i : (DN_i \cdot D_{\cM} Y)  - Q\int_\cM \langle D_Y H, \etaa\circ N\rangle +{\rm Err}_2
+{\rm Err}_3\, .
\end{align*}

\medskip

\noindent{\it Step 3. Estimate on $I_3$.} From the proof of Theorem~\ref{t:taylor_area}, 
(cf.~\eqref{e:Taylorone} and \eqref{e:taylorone2}) we conclude
$\left|1- |(DF_i)_\sharp \vec{e}|\right|\leq C \left( |DN|^2  + |A| |N|\right)$. 
To show that $I_3$ can be estimated with ${\rm Err}_2$ and ${\rm Err}_3$ observe that, 
by \eqref{e:zeta_i} we have 
\[
\left|\langle \partial_\eps \vec{v} (0,x), (DF_i)_\sharp \vec{e}\,\rangle\right| = |\zeta_i (x)| 
\leq C |DN|^2 |DY| + C |A| |DY| |N| + C|A| |DN| |Y|\, .
\]

%\begin{equation*}%\label{e:I_3}
%|I_3| \leq C \int_\cM \Big(|DY| \big(|A||N| |DN|^2 + |DN|^4 + |A|^2|N|^2\big) + |Y| \big(|A|^2|N||DN| +|A| |DN|^3\big)\Big)\, .
%\end{equation*}

\section{Reparametrizing multiple valued graphs}\label{s:reparametrize}

In this section we exploit the link between currents
and multiple valued functions in the opposite direction, in order to 
give conditions under which $Q$-valued graphs can be suitably 
reparametrized and to establish relevant estimates on the parametrization.
We fix the short-hand notation $\vec{e}= e_1\wedge \ldots \wedge e_{m+n}$,
$\vec{e}_m= e_1\wedge \ldots \wedge e_m$ and 
$\vec{e}_n= e_{m+1}\wedge \ldots \wedge e_{m+n}$, where
$e_1, \ldots, e_m, e_{m+1}, \ldots, e_{m+n}$ is the standard basis of $\R^m \times \R^n$. We will often use
the notation $\pi_0$ and $\pi_0^\perp$ for $\R^m\times \{0\}$ and $\{0\}\times \R^n$.

\begin{theorem}[$Q$-valued parametrizations]\label{t:cambio}
Let $Q,m,n\in \N$ and $s<r<1$.
Then, there are constants $c_0, C>0$ (depending on $Q,m,n$ and $\frac{r}{s}$) with the following property.
Let $\phii$, $\cM$ and $\bU$ be as in Assumption~\ref{i:tripla_malefica} with
$\Omega = B_{s}$ and let $f: B_{r} \to \Iqs$ be such that
\begin{equation}\label{e:bounds_phi_f}
\|\phii\|_{C^2} + \Lip (f)\leq c_0 \qquad \mbox{and} \qquad \|\phii\|_{C^0} + \|f\|_{C^0} \leq c_0\, r .
\end{equation}
Set $\Phii (x) := (x, \phii (x))$.
Then, there are maps $F$ and $N$ as in Assumption~\ref{i:tripla_malefica}(N) such that
$\bT_F = \bG_f \res \bU$ and
\begin{gather}
\Lip (N) \leq C \big(\|D^2\phii\|_{C^0} \|N\|_{C^0} + \|D\phii\|_{C^0} +\Lip (f)\big)\, ,\label{e:stime_cambio2}\\
\frac{1}{2\sqrt{Q}} |N (\Phii (p))| \leq \cG (f (p), Q\a{\phii (p)})\leq 2 \sqrt{Q} \, |N (\Phii (p))|\qquad 
\forall p\,\in B_s\label{e:stime_cambio1}\, , \\
|\etaa \circ N (\Phii (p))| \leq C |\etaa\circ f (p) - \phii (p)| + C \Lip (f) |D\phii (p)| |N(\Phii (p))|\, 
\qquad \forall p\,\in B_s.\label{e:stima_media}
\end{gather} 
Finally, assume $p \in B_s$ and
$(p, \etaa\circ f (p)) = \xi + q$ for some $\xi\in \cM$ and $q\perp T_\xi \cM$.
Then,
\begin{equation}\label{e:stima_cambio_10}
\cG (N(\xi), Q\a{q}) \leq 2\sqrt{Q} \,\cG (f(p), Q\a{\etaa\circ f(p)})\, . 
\end{equation}
\end{theorem}

For further reference, we state
the following immediate corollary of Theorem~\ref{t:cambio}, corresponding to the case
of a linear $\phii$.

\begin{proposition}[$Q$-valued graphical reparametrization]\label{p:cambio_lineare}
Let $Q, m, n\in \mathbb N$ and $s<r<1$. There exist positive constants $c,C$ (depending only on
$Q,m,n$ and $\frac{r}{s}$) with the following
property. Let $\pi_0$ and $\pi$ be $m$-planes
with $|\pi-\pi_0|\leq c$ and $f:B_{r} (\pi_0)\to \Iq (\pi_0^\perp)$ with 
$\Lip (f) \leq c$ and $|f|\leq c r$. Then, there is a Lipschitz
map $g: B_s (\pi)\to \Iq (\pi^\perp)$ with $\bG_g = \bG_f \res \bC_s (\pi)$ and
such that the following estimates hold on $B_s (\pi)$:
\begin{gather}
\|g\|_{C^0} \leq C r |\pi-\pi_0| + C\|f\|_{C^0},\\
\Lip (g)\leq C |\pi-\pi_0| + C \Lip (f)\, .
%|\etaa \circ g (x)| \leq C |\etaa \circ f (\p_{\pi_0} (x))| + C \, \Lip (f) \, 
%|\pi-\pi_0| \, \cG \big(f(\p_{\pi_0} (x)), Q \a{x-\p_{\pi_0} (x)}\big)\, .
%\label{e:media_puntuale}
\end{gather}
\end{proposition}

In fact the proof of Theorem \ref{t:cambio} will give
a more precise information about the map $F$,
namely its pointwise values can be determined with a
natural geometric algorithm. 
%In order to state
%this fact more precisely, we fix the following terminology.

\begin{definition}[Multiplicity in $Q$-valued maps]\label{d:multiplicity}
Given a $Q$-valued map $F$, we say that
a point $p$ {\em has multiplicity $k$} in $F(x)$ if we can write 
$F(x) = k \a{p} + \sum_{i=1}^{Q-k} \a{p_i}$ where $p_i\neq p$ for every $i$, i.e.
if $p$ has multiplicity $k$ when treating $F(x)$ as a $0$-dimensional integral current.
\end{definition}

\begin{lemma}[Geometric reparametrization]\label{l:algoritmo_naturale}
The values of $F$ in Theorem~\ref{t:cambio} can be determined at any point $p\in\cM$ as follows.
Let $\varkappa$ be the orthogonal complement of $T_p \cM$. Then,
$\gr (f) \cap (p+\varkappa)$ is nonempty, consists of at most $Q$ points and
every $q\in \gr (f) \cap (p+\varkappa)$ has in $F(p)$ the same multiplicity of
$\p_{\pi_0^\perp} (q)$ in $f (\p_{\pi_0} (q))$.
\end{lemma}

\subsection{Existence of the parametrization}
The next lemma is a natural outcome of the Ambrosio-Kirchheim approach
to the theory of currents \cite{AK}. Following \cite[Section 4.3]{Fed}, 
if $T$ is a flat $m$-dimensional current in $U$
and $h: U \to \R^k$ a Lipschitz map with $k\leq m$, we denote by $\langle T, h, y\rangle$
the slice of $T$ with respect to $h$ at the point $y$ (well-defined for a.e. $y\in \R^k$).
Since we deal with {\em normal} currents, the equivalence of the 
classical Federer-Fleming theory and the modern Ambrosio-Kirchheim theory
(cf.~\cite[Theorem 11.1]{AK})
allows us to use all the results of the paper \cite{AK}.

\begin{lemma}\label{l:AK->FF}
Consider a $C^2$ injective open curve $\gamma:]a,b[\to\R^N$, $\ell=\gamma(]a,b[)$,
a regular tubular neighborhood $\bU (\ell)$ and
the map $\q:= \gamma^{-1}\circ \p$, where $\p$ is the associated $C^1$ normal projection
$\p: \bU (\ell) \to \ell$.
Let $T$ be an integral
$1$-dimensional current in $\bU (\ell)$ with $\partial T = 0$ such that,
for a.e.~$p\in ]a,b[$,
the slice $F (p) := \langle T, \q, p\rangle$ is a sum of $Q$ (not necessarily distinct) Dirac masses $\a{P_i}$.
If the measure $\mu (A) := \|T\| (\q^{-1} (A))$ is absolutely continuous, then 
$F  \in W^{1,1}(]a,b[,\Iq (\R^N))$ in the sense of
\cite[Definition~0.5]{DS1} and $\cG (F(p), F(p')) \leq C \mu ([p, p'])$ for a.e. $p, p'$. 
\end{lemma}

\begin{proof} Consider the metric space $\bI_0$ of $0$-dimensional integral currents endowed with the
flat norm $\bF$ as defined in \cite[Section 7]{AK}. By \cite[Proof of Theorem 8.1]{AK} 
the map $p\mapsto F (p)$ is a $\bI_0$-valued function
of bounded variation in the sense of \cite[Definition 7.1]{AK}, that is:
\begin{itemize}
\item there is a countable dense set $\mathcal{F} \subset \bI_0$ such that,
for every $S\in \mathcal{F}$,
the map $\Phi_S(p):= \bF (S, F(p))$ is a real-valued function of bounded variation;
\item $|D \Phi_S| (A) \leq C \Lip (\q) \|T\| (\q^{-1} (A)) + C \|\q\|_{C^0} \|\partial T\| (\q^{-1} (A))$ for
every Borel set $A$ and a dimensional constant $C$.
\end{itemize}
On the other hand, $\partial T = 0$ and 
the measure $A\mapsto \mu (A) := \|T\| (\q^{-1} (A))$ is 
absolutely continuous with respect to the Lebesgue measure.
By a simple density argument, it holds
\begin{equation}\label{e:bv}
\left\vert \Phi_S(p) - \Phi_S(q) \right\vert \leq C\,\mu([p,q])
\quad\text{$\forall\;S \in \bI_0$ and a.e.~$p,q\in ]a,b[$}.
\end{equation}
Observe that by assumption $F(p)$ takes values in $\Iq (\R^N)$ for a.e. $p$ and,
for $S = \sum_i \a{S_i}, R = \sum_i \a{R_i} \in \Iq (\R^N)$, it is well known that
\[
\bF (S, R) = \min_{\pi\in \mathscr{P}_Q} \sum_i |S_i - R_{\sigma (i)}| \leq \cG (S, R)\, \leq C \bF(S,R).
\]
Then, it follows from \eqref{e:bv} that $|\cG(S,F(p)) - \cG(S,F(q))|\leq C\bF(F(p),F(q))\leq C\,\mu([p,q])$ for every $S\in \Iq(\R^N)$.
By \cite[Definition 0.5]{DS1}, this concludes the proof.
\end{proof}

The lemma can be used to infer, in a rather straightforward way, the existence of the
parametrization $F$ in Theorem~\ref{t:cambio}

\begin{proof}[Proof of Theorem \ref{t:cambio}: Part I] After rescaling we can assume, without loss
of generality, $r=1$. This also easily shows that the constants depend only on the ratio $\frac{r}{s}$.
We start with a procedure to identify the $Q$-valued
function $F$.
By \eqref{e:bounds_phi_f}, $\bG_f\res (B_1\times \R^n)$ 
must be supported in a neighborhood of size $4\,c_0$ of $\Phii(B_1)$.
Therefore, if the constant $c_0$ is chosen accordingly, the boundary of
$T := \bG_f\res \p^{-1} (\cM)$ is actually supported in $\p^{-1} (\partial \cM)$ and the constancy theorem gives 
$\p_\sharp T = k \a{\cM}$ for some $k\in \Z$.
First we show that $k=Q$. Consider the functions $\phii_t := t\phii$
for $t\in [0,1]$,
the manifolds $\cM_t := \gr (\phii_t)$ and 
the corresponding projections $\p_t$.
It is simple to verify that the map 
\[
t\mapsto S_t := (\p_t)_\sharp \left(\bG_f \res (\p_t^{-1} (\cM_t))\right)\, 
\]
is continuous in the space of currents. The constancy theorem gives $S_t = Q (t) \a{\cM_t}$
for some integer $Q(t)$
% because it is simply checked that $\|\bG_f\| (\partial \p_t^{-1} (\cM_t)) =0$. 
and since $S_0 = Q \a{\R^m\times \{0\}}$, it follows that
$S_1 = \p_\sharp T = Q\a{\cM}$.

Define for simplicity $\cM\ni q\mapsto T_q:=\langle \bG_f, \p, q\rangle$.
The integer rectifiable current $\bG_f$ is represented by the triple
$(\im (G), \tau, \Theta)$ as in Proposition~\ref{p:repr_formula}.
The slicing theory gives then the following properties for $\cH^m$-a.e.~$p\in \cM$
(see \cite[4.3.8]{Fed}):
\begin{itemize}
\item[(i)] $T_p$ consists of a finite sum of Dirac masses
$\sum_{i=1}^{N_p} k_i \delta_{q_i}$;
\item[(ii)] $q_i\in \gr(f)$ and $|k_i|=\Theta (q_i)$ for every $i$;
\item[(iii)] if $\vec{v}$ is the continuous unitary $m$-vector orienting $\p^{-1} (p)$
compatibly with the orientation of $\cM$, the sign of $k_i$ is 
$\textup{sgn}(\langle \vec{T} (q_i) \wedge
\vec{v} (q_i), \vec{e} \, \rangle)$.
\end{itemize}
By the bounds on $\phii$ and $f$, $\vec{T} (x)$ is close to $\vec e_m$, while
$\vec{v}$ is close to $\vec{e_n}$. 
Therefore, each $k_i$ turns out to be positive. 
On the other hand, since $\p_\sharp T = Q\a{\cM}$, then
$\sum_i k_i =Q$. This shows that $p\mapsto F(p):=\sum k_i \a{q_i}$
defines a $Q$-valued function.

Next we show the Lipschitz continuity of $F$.
Fix a coordinate direction in $\R^m$, without loss of generality $e_1$, 
and consider the map $\bU \ni z\mapsto \Lambda (z) := P \circ \p(z)$, where
$P:\R^{m+n} \to \R^{m-1}$ is the orthogonal projection
$P(x_1, \ldots, x_{m+n}) = (x_2, \ldots, x_m)$.
Consider the corresponding slice
$\tilde{T}_{\bar y}:=\langle T, \Lambda, \bar{y}\rangle$ for $\bar y \in \R^{m-1}$.
For $\cH^{m-1}$-a.e.~$\bar{y}\in P (\cM)$, $\tilde{T}_{\bar y}$ is a rectifiable 
$1$-dimensional current with $(\partial \tilde{T}_{\bar y})\res \bU=0$
(see \cite[Section 4.3.1]{Fed}).
If we slice further $\tilde{T}_{\bar y}$ with respect to the map $\p_{\bar{y}} := x_1\circ \p$,
we conclude that
for a.e.~$\bar{y}$ and a.e.~$p\in \ell_{\bar y}$
we must have $\langle \tilde{T}_{\bar y}, \p_{\bar{y}}, p\rangle
= F(p)$ (cf.~\cite[Lemma 5.1]{AK}). Applying the coarea formula to the rectifiable
set $\bG_f$ shows also that, if $c_0$ is sufficiently small, then $\|T\| (\p_{\bar{y}}^{-1} (A))
\leq C |A|$, where $C$ is a geometric constant (and $|\cdot|$ denotes the Lebesgue $1$-dimensional measure);
cf.~\cite[Theorem 4.3.8]{Fed}. Define $]a,b[ = \{t: (t,\bar{y})\in B_s\}$, $\ell :=
\{\phii (t, \bar{y}): t\in ]a,b[\}$ and $\gamma (t):= \phii (t, \bar{y})$ 
It is easy to see that on $\supp (\tilde{T}_{\bar y}$ the
map $\p_{\bar{y}}$ coincides with the map $\q$ of Lemma~\ref{l:AK->FF}. Therefore
the map $]a,b[\to F (t, \bar{y})$ is Lipschitz (up to a null-set).
Arguing in the same way
for each coordinate, we conclude that one can 
redefine $F$ on a set of measure zero in such a way that $F$ is Lipschitz: we will keep
the notation $F$ for such Lipschitz map.

Define next $N (x) = \sum_i \a{F_i (x)-x}$. We then see that, by construction,
$N$ satisfies Assumption \ref{i:tripla_malefica}(N). Fix next coordinates on $\cM$
(for instance using $\Phii$ as chart). By Proposition~\ref{p:repr_formula}
and the bounds
on $f$ and $\phii$, we deduce that
\[
\langle d \p , \vec{\bG}_f\rangle \geq c >0 \quad \text{and}\quad
\langle d \p , \vec{\bT}_F\rangle \geq c >0,
\]
for a suitable
geometric constant $c$ (where we use the notation $d\p = d\p^1 \wedge \ldots \wedge d\p^m$
and $\p^1, \ldots, \p^m$ are the components of $\p$ in the particular chart chosen on $\cM$).
Hence, if $\bT_F \neq \bG_f\res \p^{-1} (\cM)$, then necessarily
$\bT_F \res d\p \neq \bG_f \res d\p$, which is a contradiction to
$\langle T', \p, y\rangle = \langle T, \p, y\rangle$ for a.e. $y$ (cf.~\cite[(5.7) and
Theorem~5.6]{AK}).

\medskip

{\em Part II}. To prove \eqref{e:stime_cambio2}
consider first pairs of points $p, q\in \cM$ with the following property:
\begin{itemize}
\item[(AE)] let $\sigma = \Phii([\p_{\pi_0}(p),\p_{\pi_0}(q)])$,
$F|_{\sigma} = \sum \a{F_i}$ with each $F_i$ Lipschitz
(cf.~\cite[Proposition~1.2]{DS1}), and consider
the corresponding curves $\gamma_i=F_i(\sigma)$: then, for $\cH^1$-a.e.~$y\in \gamma_i$, 
$\vec\gamma_i (y)$ belongs to the tangent plane $T_y \bG_f$.
\end{itemize}
We claim that (AE) implies:
\begin{equation}\label{e:stima_intermedia}
|N(p) - N(q)| \leq C ( \|D^2 \phii\|_{C^0} \|N\|_{C^0} + \Lip (f) + \|D \phii\|_{C^0})\,
|\p_{\pi_0}(p) - \p_{\pi_0}(q)|\, .
\end{equation}
By standard measure theoretic arguments, (AE) holds for a set of pairs $(p,q)$ of full measure
in $\cM\times \cM$. With a simple density
argument we then conclude the validity of \eqref{e:stima_intermedia} for every pair $p,q$.
Denote by $d$ the geodesic distance on $\cM$. Since $|\p_{\pi_0}(p) - \p_{\pi_0}(q)| \leq d (p,q)$,
we then conclude the Lipschitz estimate \eqref{e:stime_cambio2}.

Let us turn to \eqref{e:stima_intermedia}. We parameterize $\sigma$ by arc-length
$s : [0, \ell]\to \sigma$ and for every $i$ define $n (t) := F_i (s(t))-s(t)$. 
Clearly, $n$ is Lipschitz and we claim that:
\begin{equation*}%\label{e:stima_vera}
|n' (t)| \leq C \big( \|D^2 \phii\|_{C^0} \|n\|_{C^0} + \Lip (f) + \|D \phii\|_{C^0}\big) \quad \mbox{for a.e.~$t$}.
\end{equation*}
Observe that $\frac{s' (t) + n'(t)}{|s'(t)+n'(t)|} = \vec\gamma_i (F_i(s(t)))$ which, for a.e.~$t$, belongs to $T_{F_i(s(t))}\gr(f)$.
The angle $\theta$ between $\vec\gamma_i (F_i(s(t))$ and the plane $\p^{-1} (s(t))$
can then be estimated by 
\begin{equation}\label{e:angle}
\left|\frac{\pi}{2} - \theta\right|\leq C \big(\Lip (f) + \|D\phii\|_{C^0}\big)\, .
\end{equation}
Let $\p^T$ and $\p^\perp$ be the projections to the tangent and normal planes to $\cM$ in $F_i(s(t))$.
Then, if $c_0$ is chosen small enough to have $|n'(t)| \leq 1$, we get
\begin{align}
|\p^\perp (n'(t))| &= |\p^\perp (n'(t)+s'(t))| = |n'(t)+s'(t)| |\p^\perp (\vec\gamma_i (F_i(s(t)))| \leq 2 \, 
|\cos \theta|\nonumber\\
&\stackrel{\mathclap{\eqref{e:angle}}}{\leq} C \big(\Lip (f) + \|D\phii\|_{C^0}\big)\, .\label{e:angolo}
\end{align}
In order to compute the tangential component,
% 
% \[
% s'(t) + n' (t) = \pi^t(s'(t) + n'(t)) + \pi^n (s'(t) + n'(t)) = s'(t) + \pi^t (n'(t)) + \pi^n (n'(t))\, .
% \]
let $\nu_1, \ldots, \nu_n$ be an orthonormal frame on the normal bundle. It can be chosen so that
$\|D \nu_j\|_{C^0}\leq C \|D^2\phii\|_{C^0}$ for every $j$ (see Lemma \ref{l:trivialize}).
From $n (t) := \sum_j \lambda_j (t) \nu_j (s(t))$, with $\lambda_j (t) := n (t) \cdot \nu_j (s(t))$ Lipschitz functions, we get
\[
\p^T (n'(t)) = \p^T \left( \sum \lambda_j' (t) \nu_j (s(t)) + \sum \lambda_j (t) \frac{d}{dt} \nu_j (s(t))\right)
= \sum \lambda_j (t) \p^T \left(\frac{d}{dt} \nu_j (s(t))\right)\, ,
\]
which implies
\begin{equation}\label{e:seconda_forma_3}
|\p^T (n'(t))| \leq \sum \|\lambda_j\|_{C^0} \|D\nu \|_{C^0}\leq C \|n\|_{C^0}\|D^2 \phii\|_{C^0}\, .
\end{equation}
% Thus, if $c_0$ is sufficiently small, we have $|s'(t) + \pi^t (n'(t))|\leq 2$ and
% \begin{align}
% |\pi^n (n'(t))| &= |\pi^n (n'(t)+s'(t))| = |n'(t)+s'(t)| |\pi^n (\vec\gamma_i (F_i(s(t)))| \leq 2 \, |\cos \theta|\nonumber\\
% &\stackrel{\mathclap{\eqref{e:angle}}}{\leq} C (\Lip (f) + \|D\phii\|_{C^0})\, .\label{e:angolo}
% \end{align}
Putting together \eqref{e:seconda_forma_3} and \eqref{e:angolo}, we get \eqref{e:stima_intermedia}. 
\end{proof}

\subsection{Validity of the geometric algorithm} Before completing the proof of Theorem \ref{t:cambio}
we show Lemma \ref{l:algoritmo_naturale}, which indeed will be used the derive the remaining estimates
in Theorem \ref{t:cambio}. 

\begin{proof}[Proof of Lemma~\ref{l:algoritmo_naturale}].
%Consider the two sets $A:= \gr (f) \cap \p^{-1} (\cM)$ and $B:= \im (F)$. By construction and
%the area formula, we know that $B\setminus A$ is a set of $\cH^m$-measure zero.
By the representation formula in Proposition~\ref{p:repr_formula},
since the support of the push-forward via a Lipschitz map is the image of the map and
we already proved $\bT_F = \bG_f \res \bU$, we then conclude that
$\im(F)=\gr(f)\cap \bU$ as {\em sets}.
Thus, to complete the proof of
Lemma \ref{l:algoritmo_naturale} we just have to show the rule for determining the multiplicity of
a point $q\in (p+\varkappa) \cap \gr (f)$ in $F(p)$.
This rule follows easily
from the area formula when $\Lip (f)$, $\Lip (N)$ and $\Lip (\phii)$ are smaller than a geometric constant,
since under such assumption the Taylor expansions for the mass given by Theorem \ref{t:taylor_area}
and Corollary \ref{c:taylor_area} imply the following facts:
\begin{itemize}
\item if $y$ has multiplicity $k$ in $f(x)$, then 
\begin{equation*}
k- \frac{1}{2} \leq \liminf_{r\downarrow 0} \frac{\|\bG_f\| (\bB_\rho ((x,y)))}{\omega_m \rho^m}
\leq \limsup_{r\downarrow 0} \frac{\|\bG_f\| (\bB_\rho ((x,y)))}{\omega_m \rho^m} \leq k + \frac{1}{2}\, ;
\end{equation*}
\item if $p$ has multiplicity $k$ in $F(x)$, then 
\begin{equation*}
k- \frac{1}{2} \leq \liminf_{r\downarrow 0} \frac{\|\bT_F\| (\bB_\rho (p))}{\omega_m \rho^m}
\leq \limsup_{r\downarrow 0} \frac{\|\bT_F\| (\bB_\rho (p))}{\omega_m \rho^m} \leq k + \frac{1}{2}\,.\qedhere
\end{equation*}
\end{itemize}
\end{proof}

We can now conclude the proof of Theorem~\ref{t:cambio}.

\begin{proof}[Proof of Theorem \ref{t:cambio}: Part III]
We first deal with \eqref{e:stime_cambio1} and \eqref{e:stima_media}.
Observe first that, thanks to Lemma~\ref{l:algoritmo_naturale},
the value of $N$ at the point $(p, \phii (p))$ does not change if we replace $\phii$
with its first order Taylor expansion.
Moreover, upon translation we can further assume $p=0$ and $\phii (0)=0$.
We moreover fix the notation $\pi := \{(x, D\phii (x) \cdot x):x\in\pi_0\} = T_0 \gr (\phii)$
and denote by $\varkappa$ the orthogonal complement of $\pi$. With a slight abuse of notation,
the same point $p\in \R^{m+n}$
is then represented by a pair $(x,y)\in \pi_0\times \pi_0^\perp$ and a pair $(x',y')\in \pi\times \varkappa$. 
Concerning \eqref{e:stime_cambio1}, since the role of the two systems can be reversed, it suffices to show 
only one inequality, namely
\begin{equation}\label{e:stima_in_0}
|f| (0) \leq 2\sqrt{Q} |N (0)|\, .
\end{equation}
Let $f(0) = \sum_i \a{P_i}$, $q_i := \p_\pi (P_i)$ and $N (q_i) = \sum_j \a{Q_{i,j}}$. 
There is then a $j(i)$ such that $(q_i, Q_{i,j(i)})\in \pi\times \varkappa$ 
is the same point as $(0, P_i)\in \pi_0\times \pi_0^\perp$.
Observing that $|q_i|\leq C \|D\phii\|_0 |P_i|$, we then get
\begin{align}
|P_i| &\leq |q_i| + |Q_{i,j(i)}| \leq |q_i| + |N (0)| + \cG (N (0), N(q_i)) \leq |N (0)| + (1+ \Lip (N)) |q_i|\nonumber\\
&\leq |N(0)| + C (1+ \Lip (N)) \|D\phii\|_0 |P_i|\, .
\end{align}
We use now \eqref{e:stime_cambio2} with $\phii$ linear: $\Lip (N)\leq C (\|D\phii\|_0 + \Lip (f)) \leq C c_0$. We thus conclude
\[
|P_i| \leq |N(0)| + C (1+c_0\,C) c_0 |P_i|\, .
\]
However, the constant $C$ in the last inequality is only geometric and does not depend on $c_0$. Thus, if $c_0$ is chosen sufficiently small, we conclude $|P_i|\leq 2 |N(0)|$. Summing upon $i$, we then reach $|f(0)|\leq 2Q^{\frac{1}{2}} |N(0)|$.

We now pass to \eqref{e:stima_media}, keeping the assumption
$f(0) = \sum_i \a{P_i}$ and writing $N (0) = F(0) = \sum_i \a{p_i}$.
Set $\p_{\pi_0} (p_i) = (x_i, 0) $ and $\p_{\pi_0^\perp} (p_i) = (0, y_i)$.
The angle $\theta$ between $p_i$ and $\p_{\pi_0^\perp} (p_i)$ 
is estimated by $C\,|D \phii (0)|$, because the $p_i$'s are elements of $\varkappa$. Thus,
\begin{equation}\label{e:spostamento}
|x_i|\leq |p_i \sin \theta| \leq C \,|D \phii (0)|\,|N(0)| = :\rho.
\end{equation}
Consider also that
$\p_{\pi_0^\perp}: \varkappa \to \pi_0^\perp$ is a linear invertible map and in fact we can assume that the operator
norm of its inverse, which we denote by $L$, is bounded by $2$. Thus $|\etaa \circ N (0)| \leq 2 |\sum_i y_i|$ and it suffices to estimate
\begin{equation}\label{e:verticale}
\left|\sum y_i\right| \leq \left|\sum P_i\right| + C \,\Lip (f) \rho\, .
\end{equation}
To this aim, we notice that, if we set $h = \Lip(f)\,\rho$, we can decompose $f(0)$ as
$f(0) = \sum_j \a{T_j}$ (where $T_j \in \I{Q_j}$ and $Q_1+\ldots + Q_J=Q$) so that
\begin{itemize}
\item[(i)] $d(T_j) \leq 4\,Q\,h$, where $d(S):=\max_{i,j}|s_i-s_j|$ is
the {\em diameter} of $S=\sum_i\a{s_i}$ -- cf.~\cite{DS1};
\item[(ii)] $ |z-w| > 4\,h$ for all $z \in T_j$ and $w\in T_i$ with $i\neq j$.
\end{itemize}
To prove this claim we order the $P_i$'s and partition them in
subcollections $T_1, \ldots, T_k$ with the following algorithm. $T_1$ contains $P_1$ and any other point
$P_\ell$ for which there exists a chain $P_{i(1)}, \ldots, P_{i(l)}\in \supp(T)$ of points
with $i(1) = 1$, $i(l) =\ell$ and $|P_{i(l)} - P_{i(l-1)}|\leq
4\,h$. Clearly $d(T_1) \leq 4\,Q\,h$ and if $\supp (T) = \supp (T_1)$ we are finished. Otherwise
we use the procedure above to define $T_2$ from $\supp (T)\setminus \supp (T_1)$, observing that
$|q-p|> 4\,h$ for any pair of elements $q\in \supp (T_1)$ and $p\in \supp (T)\setminus
\supp(T_1)$.

By the choice of the constants, it then follows that the function $f$
``separates'' into $J$ Lipschitz functions $f_j: B_\rho\to \I{Q_j}
(\R^{n})$ with $f (x) = \sum_{j=1}^J\a{f^j (x)}$ and $\Lip (f_j) \leq \Lip (f)$.
Consider the corresponding graphs $\gr (f^j)$. Observe that, by the geometric algorithm,
$N(0)$ contains points from each of these sets and moreover such points have, in $N(0)$,
the same multiplicity that they have in $f^j$. This means that the points $p_i$ such that
$N(0) = \sum_i \a{p_i}$ can actually be also grouped in $J$ families $\{p^j_1, \ldots , p^j_{Q_j}\}$
so that
$N(0) = \sum_{j=1}^J \sum_{l=1}^{Q_j}\a{p^j_l}$.

Note that, by the definition of the distance $\cG$, for each $p^j_l \in \supp(N(0))$ there exists
a point $P_{k(j,l)}\in \supp f^j(0)$ such that $|y^j_l -
P_{k(j,l)}|\leq \cG (f^j (\p_{\pi_0} (p^j_l)), f^j (0))\leq \Lip (f) |\p_{\pi_0} (p^j_l)| \leq h$. Thus
\begin{align*}
\left\vert\sum_i y_i \right\vert={}&\left\vert
\sum_{j=1}^J\sum_{l=1}^{Q_j} y^j_l\right\vert 
\leq 
\left\vert\sum_{j=1}^J\sum_{l=1}^{Q_j}P^j_l\right\vert 
+ \sum_{j=1}^J\sum_{l=1}^{Q_j}  |y^j_l - P^j_l |\\
\leq{}& \left\vert\sum_{i}P_i\right\vert 
%Q\,|\etaa\circ f(0)| +
+\sum_{j=1}^J\sum_{l=1}^{Q_j} \Big( |y^j_l - P_{k(j,l)} | +
|P_{k(j,l)} - P^j_l |\Big)
\leq \left\vert\sum_{i}P_i\right\vert +
%Q\,|\etaa\circ f(0)| + 
C\,h.
\end{align*}

Finally, for what concerns \eqref{e:stima_cambio_10}, observe that, without loss of generality,
we can assume $q=0$ by simply shifting $\cM$ to $q+\cM$: Lemma \ref{l:algoritmo_naturale} implies
that the map $N'$ given by 
Theorem \ref{t:cambio} applied to $q+\cM$ satisfies $N' (\xi+q) = \sum_i \a{N_i (\xi) -q}$ and so
thus $\cG (N' (\xi+q), Q\a{0}) = \cG (N (\xi), \a{q})$ Assuming $q=0$
we have $\xi=(p,\etaa \circ f (p)) = (p, \phii (p))$ and thus the estimate matches the left hand side
of \eqref{e:stime_cambio1}.
\end{proof}

\appendix
\section{Trivializing normal bundles}

In this and the forthcoming papers the following procedure will be often used. Consider 
$\cM$, $\phii$ and $\Phii$ as in Assumption \ref{i:tripla_malefica}. We then construct a standard orthonormal
frame on the normal bundle of $\cM$ as follows: 
\begin{itemize}
\item[(Tr1)] we let $e_{m+1}, \ldots, e_{m+n}$ be the standard orthonormal base of $\{0\}\times \R^n$;
\item[(Tr2)] for any $p\in \cM$ we let $\varkappa_p$ be the orthogonal complement of $T_p \cM$ and
denote by $\p_{\varkappa_p}$ the orthogonal projection onto it;
\item[(Tr3)] for any $i\in \{1, \ldots, n\}$ and any $p\in \cM$ we generate
the frame $\nu_1 (p), \ldots, \nu_n (p)$ applying the Gram-Schmidt orthogonalization procedure
to $\p_{\varkappa_p} (e_{m+1}), \ldots, \p_{\varkappa_p} (e_{m+n})$.
\end{itemize}
We record then the following lemma.

\begin{lemma}[Trivialization of the normal bundle of $\cM$]\label{l:trivialize}
If $\|D\phii\|_{C^0}$ is smaller than a geometric constant, then $\nu_1, \ldots, \nu_n$ is an
orthonormal frame spanning $\varkappa_p$ at every $p\in \cM$. Consider $\nu_j$ as function of
$x\in \Omega$ using the inverse of $\Phii$ as chart.
For every $\alpha + k \geq 0$
there is a constant $C$ depending on $m,n,\alpha, k$ such that,
if $\|\phii\|_{C^{k+1, \alpha}} \leq 1$, then $\|D\nu_j\|_{C^{k, \alpha}} \leq C \|\phii\|_{C^{k+1, \alpha}}$.
\end{lemma}
\bibliographystyle{plain}
\bibliography{references-Alm}
\end{document}